\documentclass[journal]{IEEEtran}

\usepackage{hyperref,doi,url}
\hypersetup{colorlinks=true, linkcolor=blue, breaklinks=true, urlcolor=blue}

\usepackage[table]{xcolor}
\usepackage{graphicx}
\usepackage{amsmath}
\usepackage{amssymb}
\usepackage{amsthm}
\usepackage{bm}
\usepackage[compress]{cite}

\usepackage{enumitem}
\setlistdepth{5}

\newlist{myEnumerate}{enumerate}{5}
\setlist[myEnumerate,1]{label=(\arabic*)}
\setlist[myEnumerate,2]{label=(\alph*),
ref=\themyEnumeratei(\alph*)}
\setlist[myEnumerate,3]{label=(\roman*),
ref=\themyEnumerateii(\roman*)}
\setlist[myEnumerate,4]{label=(\Alph*),
ref=\themyEnumerateiii(\Alph*)}
\setlist[myEnumerate,5]{label=(\Roman*),
ref=\themyEnumerateiv(\Roman*)}

\usepackage{dsfont}
\usepackage{bbm}
\DeclareMathAlphabet{\mathbbb}{U}{bbold}{m}{n}
\usepackage{indentfirst}
\usepackage{subfigure}
\usepackage{hhline}
\usepackage{mathtools}

\usepackage{tikz}
\usepackage{pgfplots}
\usepackage[prependcaption,colorinlistoftodos]{todonotes}

\usepackage{lineno}
\usepackage{multirow}
\usepackage{amsfonts}
\usepackage{textcomp}
\usepackage{stfloats}

\usepackage{tikz}
\usetikzlibrary{arrows.meta}
\tikzset{%
	>={Latex[width=2mm,length=2mm]},
	base/.style = {rectangle, rounded corners, draw=black,
		minimum width=2cm, minimum height=1cm,
		text centered, font=\sffamily},
	activityStarts/.style = {base, fill=blue!30},
	startstop/.style = {base, fill=red!30},
	activityRuns/.style = {base, fill=green!30},
	process/.style = {base, minimum width=2.5cm, fill=orange!15,
		font=\ttfamily},
}

\usepackage{centernot}
\usepackage{version,xspace}
\usepackage{environ}
\usepackage{blkarray}
\usetikzlibrary{automata,positioning,arrows,through}

\usepackage{threeparttable}

\usepackage{float}
\usepackage{indentfirst}
\usepackage{mathrsfs}

\usepackage[linesnumbered,ruled,vlined]{algorithm2e}
\usepackage{algorithmic}

\usepackage{tabularx}
\usepackage{cuted}

\usepackage[T1]{fontenc}

\usepackage{tikz}
\usetikzlibrary{shapes,backgrounds}
\usepackage{verbatim}

\newtheorem{defi}{\textbf{Definition}}
\newtheorem{thom}{\textbf{Theorem}}

\newtheorem{rek}{\textbf{Remark}}

\newtheorem{lema}{\textbf{Lemma}}

\newtheorem{example}{\textbf{Example}}

\DeclareMathOperator*{\Supp}{\textup{Supp}}
\DeclareMathOperator*{\inft}{\textup{inft}}

\newcommand*{\I}{\textup{I}}
\newcommand*{\p}{\textup{p}}
\newcommand*{\init}{\textup{init}}
\newcommand*{\SA}{\textup{SA}}
\newcommand*{\ISA}{\textup{ISA}}
\newcommand*{\D}{\textup{D}}
\newcommand*{\g}{\textup{g}}

\begin{document}

\title{On the Detection of Markov Decision Processes}
\author{Xiaoming Duan, Yagiz Savas, Rui Yan, Zhe Xu, and Ufuk Topcu
	\thanks{Xiaoming Duan is with the Oden Institute for Computational Engineering and Sciences, The University of Texas at Austin, Austin, TX, 78712, USA. email: {\tt\small xiaomingduan.zju@gmail.com}.}
	\thanks{Yagiz Savas and Ufuk Topcu are with the Department of Aerospace Engineering and Engineering Mechanics, The University of Texas at Austin, Austin, TX, 78712, USA. email: {\tt\small\{yagiz.savas, utopcu\}@utexas.edu}.}
	\thanks{Rui Yan is with the Department of Computer Science, University of Oxford, Oxford OX1 3QD, UK. email: {\tt\small rui.yan@cs.ox.ac.uk}.}
	\thanks{Zhe Xu is with the School for Engineering of Matter, Transport, and Energy, Arizona State University, Tempe, AZ, 85287, USA. email: {\tt\small xzhe1@asu.edu}.
	}
}

\maketitle

\IEEEpeerreviewmaketitle
\begin{abstract}
We study the detection problem for a finite set of Markov decision processes (MDPs) where the MDPs have the same state and action spaces but possibly different probabilistic transition functions. Any one of these MDPs could be the model for some underlying controlled stochastic process, but it is unknown a priori which MDP is the ground truth. We investigate whether it is possible to asymptotically detect the ground truth MDP model perfectly based on a single observed history (state-action sequence). Since the generation of histories depends on the policy adopted to control the MDPs, we discuss the existence and synthesis of policies that allow for perfect detection. We start with the case of two MDPs and establish a necessary and sufficient condition for the existence of policies that lead to perfect detection. Based on this condition, we then develop an algorithm that efficiently (in time polynomial in the size of the MDPs) determines the existence of policies and synthesizes one when they exist. We further extend the results to the more general case where there are more than two MDPs in the candidate set, and we develop a policy synthesis algorithm based on the breadth-first search and recursion. We demonstrate the effectiveness of our algorithms through numerical examples.
\end{abstract}


\section{Introduction}
\paragraph*{Problem description and motivation}
We consider a finite set of Markov decision processes (MDPs) with the same state and action spaces but potentially different transition functions. These MDPs are candidate models for some controlled stochastic process of interest, but it is unknown which MDP is the ground truth model a priori. We study the detection problem where the goal is to identify the ground truth MDP model through the observed state-action sequence under some policy. The detectability of the MDP model depends crucially on the differences among the transition functions of candidate MDPs and the policy applied in the generation of the state-action sequence. We focus on the scenario where the candidate MDP models are fixed and given, and we can fully observe the states and actions. Our aim is to synthesize a policy or decide that such a policy does not exist, using which we can successfully detect the ground truth MDP model no matter which one it is in the candidate set. 

MDPs are a widely adopted formalism to model sequential decision-making processes under uncertainties~\cite{MLP:14}. The detection problem for MDPs studied in this paper is relevant in applications such as medical decision-making~\cite{LNS-DLK-BTD:21}, active intrusion detection~\cite{BW-MA-SB-UT:19}, and recommendation systems~\cite{KC-MC-DK-PN-AR:20}. In an MDP-based recommendation system~\cite{GS-DH-RIB:05, KC-MC-DK-PN-AR:20}, the MDPs model different types of customer behavior depending on their characteristics (e.g., gender or age). The states encode the customers' past purchase histories of finite length, and the actions are the items to be selected for recommendation. The recommendation system may help provide recommendations tailored to customers by 
first identifying the customer type based on the customers' purchase history and their reactions to the recommendations.

\paragraph*{Literature review} Our work has close connections with a few different topics in various areas.

{\bf{Multi-model MDPs}}: In the literature, there are several names for the model considered in this paper: hidden model MDPs \cite{IC-JC-TGM-SN-RS-OB:12}, multi-task reinforcement learning \cite{EB-LL:13},  multiple-environment MDPs~\cite{JR-OS:14}, contextual MDPs~\cite{AH-DDC-SM:15}, multi-scenario MDPs and concurrent MDPs~\cite{PB-DS:19}, latent MDPs~\cite{JK-YE-CC-SM:21}, and multi-model MDPs~\cite{LNS-DLK-BTD:21}. The authors in \cite{IC-JC-TGM-SN-RS-OB:12} model the adaptive management problems in conservation biology and natural resources management using a hidden model MDP. The authors first show that the planning problem for a finite-horizon hidden model MDP is PSPACE-complete. Then, they develop tailored, efficient algorithms for hidden model MDPs based on general-purpose algorithms for partially observable MDPs (POMDPs). Multiple-environment MDPs first appear in~\cite{JR-OS:14}, where the authors study the strategy synthesis problems for achieving reachability, safety, or parity objectives in all the MDPs that constitute the multiple-environment MDP.  Although multiple-environment MDPs can be reformulated as general POMDPs, which are computationally intractable~\cite{KC-MC-MT:16}, the authors show that many qualitative strategy synthesis problems for them can be solved efficiently, at least in the binary case where there are two MDPs in the multiple-environment MDP. In a recent study~\cite{KC-MC-DK-PN-AR:20}, the authors consider multiple-environment MDPs as a particular case of POMDPs and mixed-observability MDPs. They exploit the structure of multiple-environment MDPs to improve the computational efficiency of general-purpose algorithms for POMDPs. We note here that one key difference between multiple-environment MDPs and POMDPs is that the unobservable state in multiple-environment MDPs, which corresponds to the identity of the ground truth MDP model in the candidate set, does not change with time. Control of multi-model MDPs has been studied in~\cite{PB-DS:19} and \cite{LNS-DLK-BTD:21}, where the authors develop algorithms to construct a single policy that maximizes a weighted sum of discounted rewards for the candidate MDPs in the finite and infinite horizon, respectively. In the finite-horizon case~ \cite{LNS-DLK-BTD:21}, the authors study both history-dependent and Markovian policies and show that deterministic policies are sufficient. In the infinite-horizon case~\cite{PB-DS:19}, the authors focus on the stationary Markovian policies and show that randomization can be strictly more beneficial. Both problems are shown to be NP-hard and solved via mixed-integer programming. Finally, the works~\cite{EB-LL:13, AH-DDC-SM:15, JK-YE-CC-SM:21} consider the learning problem for latent MDPs, where various algorithms are designed to minimize the regret against a learner that knows the ground truth MDP model in episodic settings. 

In this paper, we synthesize policies for detecting the ground truth MDP model asymptotically for multi-model MDPs (MMDPs). The authors in~\cite{BW-MA-SB-UT:19} formulate a similar problem as a general POMDP and study cost-bounded policies. However, the intrinsic detectability issue has not been addressed. 

{\bf{Detection of Markov chains}}: MDPs are closely related to Markov chains (MCs) in that they turn into (possibly time-varying) MCs once a policy is fixed. Thus, our problem also connects with the detection problems for MCs \cite[Part III]{BCL:08}.

Classical results on the testing and estimation of MCs appear in~\cite{MSB:51, TWA-LAG:57}, where the goodness of fit test and estimation of transition probabilities are developed. The authors in~\cite{DK:78} establish necessary and sufficient conditions on the transition matrices of two MCs that guarantee the asymptotic perfect detection of the MCs based on the generated history. More recently, identity testing of MCs has received considerable attention in the computer science community. The problem is to decide the length of the observation needed to correctly determine whether the observed history comes from a given MC or a different MC that is a certain distance away from the given one with high probability. The authors in~\cite{CD-ND-NG:18} and \cite{YC-PLB:19} study the identity testing of a symmetric MC. The latter paper improves upon the former by making the sample complexity bound independent of the hitting times of MCs. Results on the testing of ergodic MCs recently appear in~\cite{GW-AK:20}.

The work on the detection, estimation and testing of MCs do not directly apply to MDPs since policies of MDPs play essential roles in the detection task. Moreover, there are in general infinitely many MCs that can be induced from an MDP.

{\bf{Uncertain MDPs}}: MMDPs encode a particular class of uncertainty for MDPs by introducing a finite number of transition models. A more general class of uncertainty models for MDPs considers continuous sets of possible transition probabilities. The decision-maker then seeks a policy that optimizes against the worst-case scenario. The authors in the early reference~\cite{JKS-REL:73} study the maxmin and maxmax policies for uncertain MDPs and devise policy-iteration algorithms to solve the problem. The authors in~\cite{CCW-HKE_94} propose efficient numerical algorithms based on successive approximations for uncertainties of transition probabilities described by a finite set of linear inequalities. On the theoretical side, the authors in~\cite{GNI:05} and \cite{AN-LEIG:05} show that if the uncertainties have a particular structure, i.e., satisfying the ``rectangularity" property, then the results on standard MDPs extend to the robust formulation. More recently, the authors in~\cite{WW-DK-BR:13} introduced a relaxed notion of rectangularity and show that the solution methods remain tractable under such a condition.


\paragraph*{Contributions} In this paper, we study the detection problem for MMDPs. Compared with the classical detection problems with passive observations, our formulation features an active policy synthesis component. In fact, the statistical properties of the underlying hypotheses in the MMDP detection problem depends critically on the employed policies, and we need to simultaneously resolve the detectability issue and perform the detection task through the policy design. The main contributions of this paper are as follows.
\begin{enumerate}
\item We formulate an \emph{asymptotic perfect detection} problem for MMDPs and propose to use the so-called \emph{Bhattacharyya coefficient} \cite{AB:46} as a separation measure for MDPs under a policy. We show that the Bhattacharyya coefficient has a monotonicity property with respect to the length of the observation, which provides insights for the policy synthesis problem.
\item We establish a necessary and sufficient condition for the detectability of binary MMDPs. Based on this condition, we develop a polynomial-time algorithm to decide the existence of a policy that achieves asymptotic perfect detection and synthesize a policy when one exists.
\item We extend the binary detection problem results to the general case of more than two MDPs and develop a similar algorithm for policy synthesis based on the breadth-first search and recursion.  
\end{enumerate}

\paragraph*{Organization}
We organize the rest of the paper as follows. Section~\ref{sec:preliminaries} reviews necessary terminologies for MDPs and introduces the notion of asymptotic perfect detection. We then solve the binary detection problem for MMDPs in Section~\ref{sec:binary}. The results are extended to the general case in Section~\ref{sec:general}. We demonstrate the effectiveness of our algorithms through two numerical examples in Section~\ref{sec:numerical}. Section~\ref{sec:conclusion} finally concludes the paper.

\paragraph*{Notation} Let $\mathbb{R}$, $\mathbb{R}^n$ and $\mathbb{R}^{m\times n}$ be the set of real numbers, real vectors of dimension $n$, and real matrices of size $m$ by $n$, respectively. We denote the set of non-negative integers by $\mathbb{N}_{\geq0}$. For $m,n\in\mathbb{N}_{\geq0}$, $\mathbb{N}_{m}^n$ denotes the set of integers $\{m,m+1,\cdots,n\}$ when $ m\leq n$,  and $\mathbb{N}_{m}^n=\emptyset$ when $m>n$. The probability simplex in dimension $n$ is denoted by $\Delta_n$, i.e., $\Delta_n=\{\mathbf{x}\in\mathbb{R}^n\,|\,\sum_{i=1}^n\mathbf{x}_i=1,\mathbf{x}_i\geq0\textup{ for } i\in\mathbb{N}_{1}^n\}$. The vector of $1$'s in dimension $n$ is denoted by $\mathbbb{1}_n$. For a probability mass function $p:\mathbb{N}_{1}^n\to[0,1]$, the support $\textup{Supp}(p)$ of $p$ is defined by $\textup{Supp}(p)=\{i\in\mathbb{N}_{1}^n\,|\,p(i)>0\}$. We denote the cardinality of a finite set $S$ by $|S|$. For two sets $S_1$ and $S_2$, the set difference $S_1\setminus S_2$ contains elements that are in $S_1$ but not in $S_2$. The complement $\overline{S}$ of a subset $S$ of the whole set $\Omega$ is $\overline{S}=\Omega\setminus S$.

\section{Preliminaries}\label{sec:preliminaries}
\subsection{MDP and MMDP}
We first formally define Markov decision processes (MDPs) with finite state and action spaces.
\begin{defi}[MDP]
An MDP $M$ is a tuple $M=(\mathcal{S},\mathcal{A},\delta,s_{\init})$\footnote{The reward function is omitted in the definition as it is irrelevant in our current problem. Moreover, we consider a specific initial state rather than an initial distribution over the state space for ease of exposition.}, where
\begin{enumerate}
	\item $\mathcal{S}$ is a finite set of states;
	\item $\mathcal{A}=\cup_{s\in\mathcal{S}}\mathcal{A}_s$ is the union of the finite sets of actions $\mathcal{A}_s$ available at the state $s\in\mathcal{S}$; 
	\item $\delta:\mathcal{S}\times\mathcal{A}\times\mathcal{S}\to[0,1]$ is the transition kernel defined for all $s\in\mathcal{S}$ and $a\in\mathcal{A}_s$ satisfying
	\begin{equation*}
		\sum_{s'\in\mathcal{S}}\delta(s'\,|\,s,a)=1;
	\end{equation*}
	\item $s_{\init}\in\mathcal{S}$ is the initial state.
\end{enumerate}
\end{defi} 

We will use $\delta(\cdot\,|\, s,a)$ to denote the probability distribution over the next states when taking an action $a\in\mathcal{A}_s$ at a state $s\in\mathcal{S}$. A history $h_t=(s_0,a_0,s_1,a_1,\cdots,s_t)$ of an MDP at step $t\in\mathbb{N}_{\geq0}$ is a sequence of states and actions, where $s_0=s_{\init}$ and for all $\tau\in\mathbb{N}_{0}^{t-1}$, $a_\tau\in\mathcal{A}_{s_\tau}$, $s_{\tau+1}\in\mathcal{S}$, and  $\delta(s_{\tau+1}\,|\,s_{\tau},a_{\tau})>0$. We denote the set of histories at step $t\in\mathbb{N}_{\geq0}$ by $\mathcal{H}_t$. For a history $h_t\in \mathcal{H}_t$ and $\tau\in\mathbb{N}_{0}^{t-1}$, $h_t(\tau)$ and $h_t[\tau]$ denote the state and the action at step $\tau$ in $h_t$, respectively. In particular, $h_t(t)$ is the last state of the history $h_t$. A prefix $h_\tau$ of a history $h_t$ for $\tau\in\mathbb{N}_{0}^t$ is a history that is composed of the first $2\tau+1$ elements of $h_t$.

A history-dependent randomized policy $\bm{\pi}=(\pi_0,\pi_1,\cdots)$ is a sequence of mappings where each $\pi_t$ for $t\in\mathbb{N}_{\geq0}$ is a mapping from the set of histories $\mathcal{H}_t$ to a distribution over actions, i.e., for any $h_t\in \mathcal{H}_t$, we have $\pi_t(h_t)\in\Delta_{|\mathcal{A}_{h_t(t)}|}$. We denote the probability of choosing an action $a\in\mathcal{A}_{h_t(t)}$ at the state $h_t(t)$ by $\pi_t(a\,|\,h_t)$ and the probability distribution over the actions by $\pi_t(\cdot\,|\,h_t)$. A policy $\bm{\pi}$  is deterministic if for any $t\in\mathbb{N}_{\geq0}$ and $h_t\in \mathcal{H}_t$, the mapping $\pi_t$ specifies a distribution over actions whose support contains exactly one element, i.e., $|\textup{Supp}(\pi_t(\cdot\,|\,h_t))|=1$. A policy $\bm{\pi}$ is Markovian (memoryless) if for any $t\in\mathbb{N}_{\geq0}$, the mapping $\pi_t$ depends only on the current state, i.e., for any $h_t\in \mathcal{H}_t$, we have $\pi_t(\cdot\,|\,h_t)=\pi_t(\cdot\,|\,h_t(t))$. A stationary policy $\bm{\pi}$ is a Markovian policy that is time-independent, i.e., $\bm{\pi}=(\pi,\pi,\cdots)$. A history $h_t=(s_0,a_0,s_1,a_1,\cdots,s_t)$ is compatible with a policy $\bm{\pi}$ if for any $\tau\in\mathbb{N}_0^{t-1}$ and any prefix $h_\tau$ of $h_t$, we have $a_\tau\in\textup{Supp}(\pi_{\tau}(\cdot\,|\,h_{\tau}))$. 

We next introduce the maximal end component (MEC) of an MDP, which will be used later in the paper. 
\begin{defi}[MEC {\cite[Section 10.6.3]{CB-JK:08}}]\label{def:MEC}
An end component $C$ of an MDP $M=(\mathcal{S},\mathcal{A},\delta,s_{\init})$ is a tuple $C=(\mathcal{X},\mathcal{U})$ where
\begin{enumerate}[label=(\roman*)]
\item the set of states $\emptyset\neq\mathcal{X}\subset\mathcal{S}$;
\item the set of actions $\mathcal{U}=\cup_{s\in\mathcal{X}}\mathcal{U}_{s}$ with $\mathcal{U}_s\subset\mathcal{A}_{s}$ for all $s\in\mathcal{X}$;
\item\label{itm:nextstates} for all $s\in\mathcal{X}$ and $u\in\mathcal{U}_s$, $\Supp(\delta(\cdot\,|\,s,u))\subset\mathcal{X}$;
\item\label{itm:strongconnect} for every pair of states $s, s'\in\mathcal{X}$ and $s\neq s'$, there exists a sequence of states and actions $(s_0,u_0\cdots,s_{t})$ with $t\geq1$ such that $s_0=s$, $s_t=s'$, and for all $\tau\in\mathbb{N}_0^{t-1}$, $u_\tau\in\mathcal{U}_{s_\tau}$ and $\delta(s_{\tau+1}\,|\,s_\tau,u_\tau)>0$.  
\end{enumerate}
An end component $C=(\mathcal{X},\mathcal{U})$ is maximal in $M$ if there does not exist another end component $C'=(\mathcal{X}',\mathcal{U}')$ such that $C'\neq C$, $\mathcal{X}\subset\mathcal{X}'$ and $\mathcal{U}_{s}\subset\mathcal{U}'_{s}$ for all $s\in\mathcal{X}$.
\end{defi}

For an MDP $M=(\mathcal{S},\mathcal{A},\delta,s_{\init})$, a state-action pair $(s,a)$ for $s\in\mathcal{S}$ and $a\in\mathcal{A}_{s}$ belongs to an end component $C=(\mathcal{X},\mathcal{U})$ of $M$, denoted by $(s,a)\in C$, if $s\in\mathcal{X}$ and $a\in\mathcal{U}_{s}$. A state $s\in\mathcal{S}$ is in the end component $C=(\mathcal{X},\mathcal{U})$, denoted by $s\in C$, if $s\in\mathcal{X}$. With a slight abuse of notation, we sometimes refer to the set of states in an end component $C$ simply by $C$. We denote the set of MECs of an MDP $M$ by $\mathcal{C}(M)$, which is unique and can be computed efficiently, e.g., see~\cite[Algorithm 47]{CB-JK:08} and improved algorithms in~\cite{KC-MH:11}. With all the terminologies for MDPs in place, we are now ready to define a multi-model MDP (MMDP).
\begin{defi}[MMDP]\label{def:MMDP}
	An MMDP $\mathcal{M}$ is a set of MDPs $\mathcal{M}=\{M_i\}_{i\in\mathbb{N}_{1}^N}$, where all the MDPs in $\mathcal{M}$ have the same state space, action space and initial condition, but possibly different transition kernels, i.e., for all $i\in\mathbb{N}_{1}^N$, $M_i=(\mathcal{S},\mathcal{A}, \delta_i,s_{\init})$.
\end{defi}

When controlling an MMDP $\mathcal{M}$, we do not know which MDP in $\mathcal{M}$ governs the transition dynamics a priori. Our task is to synthesize policies for $\mathcal{M}$ so that we can perfectly detect the ground truth MDP based on a single observed history.

\subsection{Asymptotically perfect detection}\label{subsec:APD}
In order to formalize the detection problem for MMDPs, we adopt the framework of Bayesian detection. In particular, in this section,  
we follow and adapt the development of asymptotic perfect detection (APD) in~\cite[Section II]{DK:78} and \cite{TTK-LAS:67}. Let $o_t=(y_0,y_1,\cdots,y_t)$ be a discrete-time observation sequence up to time $t\in\mathbb{N}_{\geq0}$ where $y_{\tau}\in\mathbb{R}^n$ for $\tau\in\mathbb{N}_{0}^t$, and $f_t(o_t)$ and $g_t(o_t)$ be the probability density functions (PDFs) of $o_t$ under hypotheses $H_1$ and $H_2$, respectively. Suppose $q$ and $1-q$ for $q\in(0,1)$ are the estimated prior probabilities for $H_1$ and $H_2$, then the maximum a posteriori (MAP) detection rule gives that 
\begin{equation}\label{eq:MAP}
\begin{cases}
\textup{decide~}H_1,\quad \textup{if~} \frac{f_t(o_t)}{g_t(o_t)}\geq\frac{1-q}{q},\\
\textup{decide~}H_2,\quad \textup{if~} \frac{f_t(o_t)}{g_t(o_t)}<\frac{1-q}{q}.
\end{cases}	
\end{equation}
If the true prior probabilities for ${H}_1$ and ${H}_2$ are $\theta$ and $1-\theta$ for $\theta\in(0,1)$, respectively, then the probability of error $P_{\textup{error}}(t,q,\theta)$ for the MAP rule~\eqref{eq:MAP} is given by
\begin{multline}\label{eq:proberror}
P_{\textup{error}}(t,q,\theta)=\theta \int f_t(o_t)\mathbf{1}_{\{\frac{f_t(o_t)}{g_t(o_t)}\leq\frac{1-q}{q}\}}(o_t)do_t
	\\+ (1-\theta) \int g_t(o_t)\mathbf{1}_{\{\frac{f_t(o_t)}{g_t(o_t)}\geq\frac{1-q}{q}\}}(o_t)do_t,
\end{multline}
where $\mathbf{1}_{\{\cdot\}}(\cdot)$ is the indicator function. We say that APD is achieved for $H_1$ and $H_2$ when the probability of error $P_{\textup{error}}(t,q,\theta)$ approaches zero for any $q\in(0,1)$ and $\theta\in(0,1)$ as $t$ approaches infinity. Let the Bhattacharyya coefficient (BC) $B(t)$ between the PDFs $f_t(o_t)$ and $g_t(o_t)$ be
\begin{equation*}
B(t)=\int\sqrt{f_t(o_t)\cdot g_t(o_t)}do_t.
\end{equation*}
We present the bounds on the probability of error and a necessary and sufficient condition for APD in terms of the BC in the following lemma.

\begin{lema}[Bounds on the probability of error and necessary and sufficient condition for APD {\cite[Eq. (3)]{TTK-LAS:67}},{\cite[Eq. (4)]{ DK:78}}]\label{lema:APDgeneral}
	Let $o_t=(y_0,y_1,\cdots,y_t)$ be the observation sequence, and $f_t(o_t)$ and $g_t(o_t)$ be the PDFs of $o_t$ under hypotheses ${{H}}_1$ and ${{H}}_2$, respectively. Then, the probability of error $P_{\textup{error}}(t,q,\theta)$ defined in~\eqref{eq:proberror} for the MAP rule~\eqref{eq:MAP} with $q\in(0,1)$ and $\theta\in(0,1)$ satisfies\footnote{The upper bound of the probability of error is slightly different from the ones in \cite{TTK-LAS:67} and \cite{ DK:78} since we consider here the  case where the estimated prior $q$ and the true prior $\theta$ are not necessarily equal to each other.}
	\begin{multline}\label{eq:boundsonB}
\frac{1}{2}\min\{\theta,1-\theta\}B(t)^2\leq P_{\textup{error}}(t,q,\theta)\\\leq\max\{\sqrt{\frac{1-q}{q}}\theta,\sqrt{\frac{q}{1-q}}(1-\theta)\}B(t).
	\end{multline}
	Moreover, the probability of error		
	 $\lim_{t\rightarrow\infty}P_{\textup{error}}(t,q,\theta)=0$ if and only if 
	\begin{equation}\label{eq:bd}
		\lim_{t\rightarrow\infty} B(t)=0.
	\end{equation}
\end{lema}

\begin{rek}[Immunity to biased estimated priors]
From~\eqref{eq:boundsonB}, we notice that even if there is a  mismatch between the estimated prior $q$ and the true prior $\theta$, the probability of error vanishes as long as the BC goes to zero. In other words, condition~\eqref{eq:bd} for APD is immune to biased estimated priors.
\end{rek}

Note that the BC is related to the perhaps more popular and well-known Hellinger distance for probability distributions~\cite{TTK-LAS:67}. We will use the BC to derive conditions for policies that achieve APD for MMDPs, where the observation in the case of MMDPs is the state-action sequence under the employed policy.

\subsection{Problem of interest}
We are interested in the APD of MMDPs. Specifically, given an MMDP $\mathcal{M}$, we develop algorithms that decide the existence of a policy that allows us to asymptotically perfectly detect the ground truth MDP in $\mathcal{M}$ based on the generated state-action sequence. Moreover, the algorithms compute such a policy when one exists. We will mainly use condition~\eqref{eq:bd} for APD in our later analysis and design.


\section{Detection of Binary MMDPs}\label{sec:binary}
In contrast to passively collecting the observation sequence generated according to candidate distributions in classical hypothesis testing tasks, in the case of MMDPs, we have the flexibility of actively taking actions at each state and observing the consequent transitions. Therefore, APD for an MMDP $\mathcal{M}$ depends crucially on the structural properties of the MDPs in $\mathcal{M}$ as well as the applied policy. In this section, we focus on the binary case where the MMDP $\mathcal{M}$ consists of two MDPs.

\subsection{Properties of the BC for binary MMDPs}\label{sec:propertiesofBC}
In MMDPs, we observe the history generated by one of the candidate MDPs, i.e., the observation sequence $o_t$ in Section~\ref{subsec:APD} becomes $h_t=(s_0,a_0,\cdots,s_t)$. Given a binary MMDP $\mathcal{M}=\{M_1,M_2\}$ and a policy $\bm{\pi}$, the BC $B(t,\bm{\pi})$ for $\mathcal{M}$ at step $t\in\mathbb{N}_{\geq0}$ under the policy $\bm{\pi}$ is then defined by 
\begin{equation}\label{eq:Bdistance}
B(t,\bm{\pi}) = \sum_{h_t\in \mathcal{H}_t}\sqrt{\mathbb{P}_1^{\bm{\pi}}(h_t)\mathbb{P}_2^{\bm{\pi}}(h_t)},
\end{equation}
where $\mathcal{H}_t$ is the union of the sets of histories of $M_1$ and $M_2$,  and $\mathbb{P}_1^{\bm{\pi}}(h_t)$ and $\mathbb{P}_2^{\bm{\pi}}(h_t)$ are the probabilities that $h_t$ occurs in $M_1$ and $M_2$ under the policy $\bm{\pi}$, respectively. We first establish useful properties of $B(t,\bm{\pi})$ for any given policy $\bm{\pi}$ in the following lemma.

\begin{lema}[Monotonicity and convergence property]\label{lema:monotone}
Given a binary MMDP $\mathcal{M}=\{M_1,M_2\}$ and a policy $\bm{\pi}$, let $B(t,\bm{\pi})$ be the BC defined in~\eqref{eq:Bdistance}. Then the following statements hold:
\begin{enumerate}[label=(\roman*)]
\item\label{itm:monotone} $B(t,\bm{\pi})$ is monotonically non-increasing, i.e.,  for all $t\in\mathbb{N}_{\geq0}$,	$B(t+1,\bm{\pi})\leq B(t,\bm{\pi})$;
\item\label{itm:convergence}  the limit $\lim_{t\rightarrow\infty}B(t,\bm{\pi})$ exists.
\end{enumerate}
\end{lema}
\begin{proof}
	Regarding~\ref{itm:monotone}, for $t\in\mathbb{N}_{\geq0}$, we expand $B(t+1,\bm{\pi})$ and obtain
\begin{align*}
	\begin{split}
B(t+1,\bm{\pi})&=\sum_{h_{t+1}\in \mathcal{H}_{t+1}}\sqrt{\mathbb{P}_1^{\bm{\pi}}(h_{t+1})\mathbb{P}_2^{\bm{\pi}}(h_{t+1})}\\
&=\sum_{h_{t}\in \mathcal{H}_{t}}\sqrt{\mathbb{P}_1^{\bm{\pi}}(h_{t})\mathbb{P}_2^{\bm{\pi}}(h_{t})}\big(\sum_{a\in\mathcal{A}_{h_t(t)}}\pi_t(a\,|\,h_t)\\
&\quad\cdot\sum_{s\in\mathcal{S}}\sqrt{\delta_1(s\,|\,h_t(t),a)\delta_2(s\,|\,h_t(t),a)}\big)\\
&\leq\sum_{h_{t}\in \mathcal{H}_{t}}\sqrt{\mathbb{P}_1^{\bm{\pi}}(h_{t})\mathbb{P}_2^{\bm{\pi}}(h_{t})}\big(\sum_{a\in\mathcal{A}_{h_t(t)}}\pi_t(a\,|\,h_t)\big)\\
&=\sum_{h_{t}\in \mathcal{H}_{t}}\sqrt{\mathbb{P}_1^{\bm{\pi}}(h_{t})\mathbb{P}_2^{\bm{\pi}}(h_{t})}=B(t,\bm{\pi}),
	\end{split}
\end{align*}
where the inequality follows from the Cauchy-Schwarz inequality and the fact that $\delta_1(\cdot\,|\,h_t(t),a)$ and $\delta_2(\cdot\,|\,h_t(t),a)$ are probability distributions, and the second to the last equality follows from the fact that $\pi_t(\cdot\,|\,h_t)$ is a probability distribution.

Regarding~\ref{itm:convergence}, note that $B(t,\bm{\pi})$ is lower bounded by zero. Then, the convergence of $B(t,\bm{\pi})$ follows from~\ref{itm:monotone} and the monotone convergence theorem \cite[Theorem 2.4.2]{SA:15}.
\end{proof}

From the proof of Lemma~\ref{lema:monotone}, we notice that $B(t+1,\bm{\pi})$ strictly decreases compared to $B(t,\bm{\pi})$ if and only if there exists at least one history $h_{t}\in \mathcal{H}_{t}$ with $\sqrt{\mathbb{P}_1^{\bm{\pi}}(h_{t})\mathbb{P}_2^{\bm{\pi}}(h_{t})}>0$ and an action $a\in\Supp(\pi_t(\cdot\,|\,h_t))$ such that the transition functions $\delta_1(\cdot\,|\,h_t(t),a)$ and $\delta_2(\cdot\,|\,h_t(t),a)$ are different. This observation is consistent with our intuition that in order to distinguish $M_1$ and $M_2$ and make $B(t,\bm{\pi})$ vanish, we should select a policy $\bm{\pi}$ under which the histories of the MDPs are statistically different.

When the length of the history is infinite, there are uncountably many possible histories and the summation in~\eqref{eq:Bdistance} should be interpreted as an integral. Specifically, let $(\mathcal{H},\mathcal{Q})$ be the measurable space where $\mathcal{H}$ is the sample space consisting of all possible infinite histories of $M_1$ and $M_2$ and $\mathcal{Q}$ is the smallest $\sigma$-algebra generated by the cylinder sets of $\mathcal{H}$ \cite{RBA-CAD:99}. Then, we have
\begin{align}\label{eq:BdistanceLeb}
	B(\bm{\pi})=\lim_{t\rightarrow\infty}B(t,\bm{\pi})=\int_{h\in \mathcal{H}}\sqrt{\mathcal{P}_1^{\bm{\pi}}(d h)\cdot\mathcal{P}^{\bm{\pi}}_2(dh)},
\end{align}
where $\mathcal{P}_i^{\bm{\pi}}$ for $i\in\{1,2\}$ is the probability measure induced by the transition kernel $\delta_i$ and the policy $\bm{\pi}$ over the measurable space $(\mathcal{H},\mathcal{Q})$. Instead of dealing with the integral~\eqref{eq:BdistanceLeb} directly, we will later work with an equivalent condition on the probability measures $\mathcal{P}_1^{\bm{\pi}}$ and $\mathcal{P}_2^{\bm{\pi}}$ such that $B(\bm{\pi})=0$.


\subsection{Informative states and state-action pairs}
As discussed in Section~\ref{sec:propertiesofBC}, the BC decreases with the length of the observation only when state-action pairs that satisfy certain properties appear in the histories of the MDPs generated under the policy. 
 This observation motivates us to define the following notions of \emph{informative} and \emph{revealing} states, actions and state-action pairs.

\begin{defi}[Informative and revealing states, actions and state-action pairs]\label{def:informative_state_action}
Given a binary MMDP $\mathcal{M}=\{M_1,M_2\}$, 
\begin{enumerate}[label=(\roman*)]
\item for a state $s\in\mathcal{S}$ and an action $a\in\mathcal{A}_s$, the pair $(s,a)$ is informative if $\delta_1(\cdot\,|\,s,a)\neq\delta_2(\cdot\,|\,s,a)$ and $\Supp(\delta_1(\cdot\,|\,s,a))\cap\Supp(\delta_2(\cdot\,|\,s,a))\neq\emptyset$; the pair $(s,a)$ is revealing if $\Supp(\delta_1(\cdot\,|\,s,a))\cap\Supp(\delta_2(\cdot\,|\,s,a))=\emptyset$;

\item a state $s\in\mathcal{S}$ is revealing if there exists an action $a\in\mathcal{A}_s$ such that $(s,a)$ is revealing, and the corresponding action is a revealing action; a state $s\in\mathcal{S}$ is informative if it is not revealing and there exists an action $a\in\mathcal{A}_s$ such that $(s,a)$ is informative, and the corresponding action is informative.
\end{enumerate} 
\end{defi}

We denote the set of informative state-action pairs in an MMDP $\mathcal{M}$ by $\ISA$, i.e., $\ISA=\{(s,a)\,|\,s\in\mathcal{S},a\in\mathcal{A}_{s}, (s,a)\textup{ is informative in } \mathcal{M}\}$. We illustrate the concepts in Definition~\ref{def:informative_state_action} via the following example.

\begin{example}[Illustrations of informative and revealing states, actions and state-action pairs]\label{exam:informative_revealing}
Consider a binary MMDP $\mathcal{M}=\{M_1,M_2\}$, where the transition diagram of the MDPs in $\mathcal{M}$ is shown in Fig.~\ref{fig:informative_state_action}. 
 \begin{figure}[http]
	\centering
	\begin{tikzpicture}
		\node[state,minimum size = 0.2cm] at (0, 0) (nodeone) {1};
		\node[state,minimum size = 0.2cm] at (2.5, 1.25)     (nodetwo)     {2};
		\node[state,minimum size = 0.2cm] at (5, 2.25)     (nodefive)     {5};	
		\node[state,minimum size = 0.2cm] at (5, 0.25)     (nodesix)     {6};	
		\node[state,minimum size = 0.2cm] at (2.5, -1.25)     (nodethree)     {3};
		\node[state,minimum size = 0.2cm] at (5, -1.25)     (nodefour)     {4};
		
		\node[state,minimum size = 0.2cm] at (7.5, 2.25)     (nodeseven)     {7};	
		
		\draw[every loop,
		auto=right,
		>=latex,
		]
		(nodeone) edge[bend left=0, auto=left] node {$ \delta_{i}(2\,|\,1,a^1)$} (nodetwo)									
		
		(nodetwo) edge[bend left=0, auto=left] node {$ \delta_{i}(5\,|\,2,\cdot)$} (nodefive)									
		(nodetwo) edge[bend left=0, auto=left] node {$ \delta_{i}(6\,|\,2,\cdot)$} (nodesix)									
		
		(nodetwo)     edge[loop below] node {$\delta_i(2\,|\,2,\cdot)$} (nodetwo)	
		(nodethree)     edge[bend left=15, auto=left] node {$\frac{1}{2}$} (nodefour)	
		(nodefour) edge[bend left=15, auto=left] node {$\frac{1}{2}$} (nodethree)						
		(nodeone) edge[bend left=0, auto=right] node {$\delta_i(3\,|\,1,a^1)$} (nodethree)						
		(nodethree)     edge[loop left] node {$\frac{1}{2}$} (nodethree)			
		(nodefour)     edge[loop right] node {$\frac{1}{2}$} (nodefour)

		(nodefive)     edge[loop above] node {$\delta_i(5\,|\,5,\cdot)$} (nodefive)	
		(nodefive)     edge[bend left=0, auto=left] node {$\delta_i(7\,|\,5,\cdot)$} (nodeseven)

		(nodesix)     edge[loop right] node {$1$} (nodesix)				
		(nodeseven)     edge[loop above] node {$1$} (nodeseven)					
		;
	\end{tikzpicture}
	\caption{The transition diagram of the MDPs in a binary MMDP $\mathcal{M}$ where there are seven states and the labels on edges represent transition probabilities between states.}\label{fig:informative_state_action}
\end{figure}
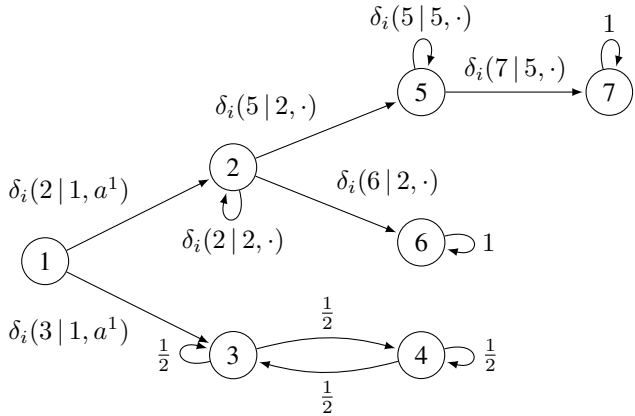

There are seven states in the state space $\mathcal{S}=\mathbb{N}_{1}^7$, and the directed edges connecting different states represent possible transitions between states after taking respective actions. We label the edges out of a state with specific transition probabilities if only one action is available at that state and the transition probabilities are the same in $M_1$ and $M_2$. We specify the rest of the transition kernels in Table~\ref{tb:tk}.
\begin{table}[http]
\centering
\begin{tabular}{|c|c|c|c|c|}
	\hline
	State&Action&Next state&$M_1$&$M_2$\\ 
	\hline
	\multirow{2}{*}{1}&\multirow{2}{*}{$a^1$}&2&0.7&0.4\\
	&&3&0.3&0.6\\\hhline{|=|=|=|=|=|}		
	\multirow{6}{*}{2}&\multirow{3}{*}{$a^2$}&2&0.2&0.2\\
	&&5&0.3&0.3\\
	&&6&0.5&0.5\\\cline{2-5}
	&\multirow{3}{*}{$b^2$}&2&0.5&0.5\\
	&&5&0.5&0\\
	&&6&0&0.5	\\\hhline{|=|=|=|=|=|}
	\multirow{4}{*}{5}&\multirow{2}{*}{$a^5$}&5&0.7&0.3\\
	&&7&0.3&0.7\\\cline{2-5}
	&\multirow{2}{*}{$b^5$}&5&1&0\\
	&&7&0&1\\\hline
\end{tabular}		
\caption{Transition kernels of the MDPs shown in Fig~\ref{fig:informative_state_action}}\label{tb:tk}
\end{table}

By Definition~\ref{def:informative_state_action}, the states $1$ and $2$ and actions $a^1$ and $b^2$ are informative, and $\ISA=\{(1,a^1),(2,b^2)\}$; the state $5$, action $b^5$ and the state-action pair $(5,b^5)$ are revealing.
\end{example}

The revealing and informative states and state-action pairs play an important role in the detection of MMDPs. At a revealing state $s\in\mathcal{S}$, the underlying MDP in an MMDP can be immediately determined by taking a revealing action $a\in\mathcal{A}_s$ and observing the consequent transition. On the other hand, the informative state-action pairs in $\ISA$ repeatedly appearing in histories make the BC decrease over time. The following lemma shows that we can focus on revealing actions at revealing states without loss of generality.

\begin{lema}[Actions at revealing states]\label{lema:executingrevealing}
Given a binary MMDP $\mathcal{M}=\{M_1,M_2\}$, let $\mathcal{S}_{\textup{r}}\subset\mathcal{S}$ be the set of revealing states in $\mathcal{M}$. For any policy $\bm{\pi}$, let $\bm{\pi}'$ is a policy such that for all $t\in\mathbb{N}_{\geq0}$ and $h_t\in \mathcal{H}_t$,
\begin{enumerate}
\item if $h_t(t)\notin\mathcal{S}_{\textup{r}}$, then $\bm{\pi}'(h_t)=\bm{\pi}(h_t)$;
\item if $h_t(t)\in\mathcal{S}_{\textup{r}}$, then $\bm{\pi}'(a\,|\,h_t)=1$ for some revealing action $a\in\mathcal{A}_{h_t(t)}$.
\end{enumerate}
Then,
\begin{equation}\label{eq:atrevealing}
B(\bm{\pi}')\leq B(\bm{\pi}).
\end{equation}
\end{lema}
\begin{proof}
Note that for any history $h_t\in \mathcal{H}_t$, if $h_t$ does not contain any revealing state $s\in\mathcal{S}_{\textup{r}}$, then $\sqrt{\mathbb{P}_1^{\bm{\pi}}(h_t)\mathbb{P}_2^{\bm{\pi}}(h_t)}=\sqrt{\mathbb{P}_1^{\bm{\pi}'}(h_t)\mathbb{P}_2^{\bm{\pi}'}(h_t)}$. However, if $h_t$ contains a revealing state $s\in\mathcal{S}_{\textup{r}}$, then $\sqrt{\mathbb{P}_1^{\bm{\pi}'}(h_t)\mathbb{P}_2^{\bm{\pi}'}(h_t)}=0$ and $\sqrt{\mathbb{P}_1^{\bm{\pi}}(h_t)\mathbb{P}_2^{\bm{\pi}}(h_t)}\geq0$. Therefore, we have that $B(t,\bm{\pi}')\leq B(t,\bm{\pi})$ for all $t\in\mathbb{N}_{\geq0}$, which leads to~\eqref{eq:atrevealing} when we take the limit $t\rightarrow\infty$.
\end{proof}

\subsection{Preprocessing of MMDPs}\label{sec:preprocessing}
By Lemma~\ref{lema:executingrevealing}, at a revealing state in a binary MMDP, we can safely ignore all other actions but one that is revealing. Moreover, for the detection of MMDPs, we can terminate the detection process immediately when identity-revealing transitions (transitions that are possible in precisely one of the MDPs) are observed. In order to simplify the analysis and policy synthesis in later sections, we propose to preprocess the MDPs in a binary MMDP in this subsection. The preprocessing removes all but one revealing action at a revealing state, introduces two special terminal states indicating successful detection, and directs identity-revealing transitions to those special states in respective MDPs. Specifically, given a binary MMDP $\mathcal{M}=\{M_1,M_2\}$ where $M_i=(\mathcal{S},\mathcal{A},\delta_i,{s}_{\init})$ for $i\in\{1,2\}$, a preprocessed MMDP $\mathcal{M}^{\p}=\{M_1^{\p},M_2^{\p}\}$ consists of MDPs $M_i^{\p}=(\mathcal{S}^{\p},\mathcal{A}^{\p},\delta_i^{\p},{s}_{\init})$ for $i\in\{1,2\}$ satisfying
\begin{enumerate}
	\item $\mathcal{S}^{\p}=\mathcal{S}\cup\{\bot_1,\bot_2
	\}$;
\item $\mathcal{A}^{\p}=\cup_{s\in\mathcal{S}}\mathcal{A}^{\p}_{s}\cup\mathcal{A}_{\bot_1}\cup\mathcal{A}_{\bot_2}$ where for $s\in\mathcal{S}$, $\mathcal{A}^{\p}_s=\mathcal{A}_s$ if $s$ is not revealing, $\mathcal{A}^{\p}_s=\{a\}$ if $s$ and $a\in\mathcal{A}_s$ are revealing, and $\mathcal{A}_{\bot_1}=\{a^{\bot_1}\}$ and $\mathcal{A}_{\bot_2}=\{a^{\bot_2}\}$ are the actions available at states $\bot_1$ and $\bot_2$, respectively; 
\item If $(s,a)$ is neither revealing nor informative  for $s\in\mathcal{S}$ and $a\in\mathcal{A}_s$, then 
$\delta_i^{\p}(s'\,|\,s,a)=\delta_i(s'\,|\,s,a)$ for all $s'\in\mathcal{S}$; if $s\in\mathcal{S}$ is revealing, then $\delta_i^{\p}(\bot_i\,|\,s,a)=1$ where $a\in\mathcal{A}^{\p}_s$; if $(s,a)$ is informative for $s\in\mathcal{S}$ and $a\in\mathcal{A}_s$, then $\delta_{i}^{\p}(s'\,|\,s,a)=\delta_{i}(s'\,|\,s,a)$ for $s'\in\Supp(\delta_{i}(\cdot\,|\,s,a))\cap\Supp(\delta_{3-i}(\cdot\,|\,s,a))$ and $\delta_{i}^{\p}(\bot_i\,|\,s,a)=\sum_{s'\in \Supp(\delta_{i}(\cdot\,|\,s,a))\setminus\Supp(\delta_{3-i}(\cdot\,|\,s,a))}\delta_{i}(s'\,|\,s,a)$; finally, $\delta_i^{\p}(\bot_j\,|\,\bot_j,a^{\bot_j})=1$ for $j\in\{1,2\}$.
\end{enumerate}

The preprocessed MMDP $\mathcal{M}^{\p}=\{M_1^{\p},M_2^{\p}\}$ of $\mathcal{M}$ is a valid MMDP since it satisfies Definition~\ref{def:MMDP}. We show the preprocessed MDP $M_1^{\p}$ of $M_1$ from Example~\ref{exam:informative_revealing} in Fig.~\ref{fig:preprocessed}. 
 \begin{figure}[http]
	\centering
	\begin{tikzpicture}
		\node[state,minimum size = 0.2cm] at (0, 0) (nodeone) {1};
		\node[state,minimum size = 0.2cm] at (2.5, 1.25)     (nodetwo)     {2};
		\node[state,minimum size = 0.2cm] at (5, 2.25)     (nodefive)     {5};	
		\node[state,minimum size = 0.2cm] at (5, 0.25)     (nodesix)     {6};	
		\node[state,minimum size = 0.2cm] at (2.5, -1.25)     (nodethree)     {3};
		\node[state,minimum size = 0.2cm] at (5, -1.25)     (nodefour)     {4};
		
		\node[state,minimum size = 0.2cm] at (7.5, 2.25)     (nodeseven)     {7};	
		
		\node[state,minimum size = 0.1cm] at (1.5, 4.5)     (nodeabs1)     {$\bot_1$};	
		
		\node[state,minimum size = 0.1cm] at (6, 4.5)     (nodeabs2)     {$\bot_2$};

		\draw[every loop,
		auto=right,
		>=latex,
		]
		(nodeone) edge[bend left=0, auto=left] node {$0.7$} (nodetwo)									
		
		(nodetwo) edge[bend left=0, auto=left] node[text width=1cm,align=center] {$ \delta_{1}(5\,|\,2,a^2)$,\\$ \delta_{1}(5\,|\,2,b^2)$} (nodefive)									
		(nodetwo) edge[bend left=0, auto=left] node[xshift=-4pt, yshift=-3pt] {$ \delta_{1}(6\,|\,2,a^2)$} (nodesix)									
		
		(nodetwo)     edge[loop below] node {$\delta_1(2\,|\,2,a^2)$} (nodetwo)	
		(nodethree)     edge[bend left=15, auto=left] node {$\frac{1}{2}$} (nodefour)	
		(nodefour) edge[bend left=15, auto=left] node {$\frac{1}{2}$} (nodethree)						
		(nodeone) edge[bend left=0, auto=right] node {$0.3$} (nodethree)						
		(nodethree)     edge[loop left] node {$\frac{1}{2}$} (nodethree)			
		(nodefour)     edge[loop right] node {$\frac{1}{2}$} (nodefour)					
		
		(nodesix)     edge[loop right] node {$1$} (nodesix)		
				(nodeseven)     edge[loop above] node {$1$} (nodeseven)							
		(nodefive)     edge[bend left=0, auto=right] node {$1$} (nodeabs1)			
		
		(nodetwo) edge[bend left=0, auto=left] node {$\delta_1(6\,|\,2,b^2)$} (nodeabs1)			
		
		(nodeabs1)     edge[loop left] node {$1$} (nodeabs1)		
		(nodeabs2)     edge[loop right] node {$1$} (nodeabs2)		
									
		;
	\end{tikzpicture}
	\caption{The preprocessed MDP $M_1^{\p}$ corresponding to $M_1$ in Example~\ref{exam:informative_revealing}.}\label{fig:preprocessed}
\end{figure}
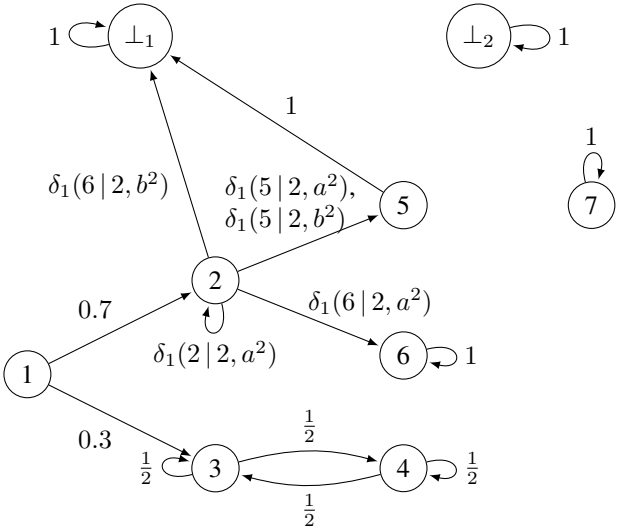

%

Since we only modify the identity-revealing transitions during the preprocessing,  the BC for $\mathcal{M}$ is equal to that for $\mathcal{M}^{\p}$ under the same policy, i.e., the detection problem for $\mathcal{M}$ is equivalent to that for $\mathcal{M}^{\p}$. 
The set of informative state-action pairs $\ISA^{\p}$ for $\mathcal{M}^{\p}$ contains the terminal states and the associated actions compared to $\ISA$ , i.e.,
\begin{equation*}
\ISA^{\p}=\ISA\cup\{(\bot_1,a^{\bot_1}),(\bot_2,a^{\bot_2})\}.
\end{equation*}

%
%
%
%


\subsection{APD for binary MMDPs}
Before presenting our policy synthesis algorithm, we further introduce the \emph{informative MDP} of a binary MMDP $\mathcal{M}$. It turns out that the policy synthesis problem for the detection of $\mathcal{M}$ can be transcribed to a problem of synthesizing policies that satisfy certain properties on the informative MDP. 

\begin{defi}[Informative MDP]\label{def:informative_MDP}
	Given a binary MMDP $\mathcal{M}=\{M_1,M_2\}$ and its preprocessed counterpart $\mathcal{M}^{\p}=\{M_1^{\p},M_2^{\p}\}$, an informative MDP $M^{\I}$ is a tuple  $M^{\I}=(\mathcal{S}^{\p},\mathcal{A}^{\p},\delta^{\I},{s}_{\init})$ where $\mathcal{S}^{\p}$, $\mathcal{A}^{\p}$, $s_{\init}$ are the same state space, action space and initial state as $M_1^{\p}$, and $\delta^{\I}=\gamma\delta_1^{\p}+(1-\gamma)\delta_2^{\p}$ for any $\gamma\in(0,1)$.
\end{defi}

In Definition~\ref{def:informative_MDP}, the informative MDP of a given binary MMDP is not unique. However, all the informative MDPs have the same transition structure, which essentially determines the solution to our policy synthesis problem for APD. 

The following theorem identifies a necessary and sufficient condition on a policy $\bm{\pi}$ for achieving APD for binary MMDPs. For an infinite history $h$ of an MDP, we will denote the set of state-action pairs that appear infinitely often (i.o.) in $h$ by $\inft(h)$, i.e., $\inft(h)=\{(s,a)\,|\,s\in\mathcal{S},a\in\mathcal{A}_{s}, \textup{ and } (s,a)\textup{ appears i.o. in } h \}$.

\begin{thom}[Necessary and sufficient condition for APD for binary MMDPs]\label{thm:main}
Given a binary MMDP $\mathcal{M}=\{M_1,M_2\}$ and any of its informative MDPs $M^{\I}$, a policy $\bm{\pi}$ achieves APD for $\mathcal{M}$ if and only if 
\begin{equation*}
\mathcal{P}_{\I}^{\bm{\pi}}(\{h\in \mathcal{H}^{\I}: \inft(h)\cap{\ISA^{\p}}\neq\emptyset\})=1, 
\end{equation*}
where $\mathcal{P}_{\I}^{\bm{\pi}}$ is the probability measure induced by the transition kernel $\delta^{\I}$ and the policy $\bm{\pi}$ over the measurable space $(\mathcal{H}^{\I},\mathcal{Q}^{\I})$ of $M^{\I}$.
\end{thom}
\begin{proof}
We postpone the proof to Appendix~\ref{sec:proofmainthm}. 
\end{proof}

Based on Theorem~\ref{thm:main}, we can transform the policy synthesis problem for APD for a binary MMDP $\mathcal{M}$ to the problem of searching for a policy that satisfies certain properties on an informative MDP of $\mathcal{M}$. 
We develop Algorithm~\ref{alg:APD} that determines the existence of a policy that achieves APD for a binary MMDP $\mathcal{M}$. The algorithm also returns a correct policy if it exists.


\begin{algorithm}[http]
	\KwIn{A binary MMDP $\mathcal{M}$}	
	\KwOut{A boolean variable indicating whether the policy for APD exists and a policy}
	\SetKwBlock{Begin}{function}{end function}
	\Begin($\texttt{BiAPD} {(} \mathcal{M} {)}$)
	{
	Construct an informative MDP $M^{\I}$ of $\mathcal{M}$\label{alg:IMDP}\\
Compute the set of MECs $\mathcal{C}(M^{\I})$\label{alg:mec}\\
 Find the set of informative MECs $\mathcal{C}^{\I}(M^{\I})$ in~\eqref{eq:informativeMEC}\label{alg:inmec}\\
 Compute the set of states that reach $\mathcal{C}^{\I}(M^{\I})$ w.p. 1: $\mathcal{R}^{\textup{max}}=\{s\in\mathcal{S}^{\p}\,|\,\mathbb{P}_{s,M^{\I}}^{\textup{max}}(\textup{reach}(\mathcal{C}^{\I}(M^{\I})))=1\}$\label{alg:setreachmax}\\
\uIf{$s_{\init}\notin\mathcal{R}^{\textup{max}}$\label{alg:maximumreach}}
	{
	\Return{$(0, \emptyset)$}
	}
\Else
{
	Synthesize a policy $\bm{\pi}^0$ such that the set of states in $\mathcal{C}^{\I}(M^{\I})$ is reached w.p. $1$ from $s_{\init}$\label{alg:reachpolicy}\\
\For{$C=(\mathcal{X},\mathcal{U})\in\mathcal{C}^{\I}(M^{\I})$}
{\label{alg:insideMEC}
	\For{$s\in\mathcal{X}$}
	{
		$\bm{\pi}^C(a\,|\,s)=\frac{1}{|\mathcal{U}_s|}$ for $a\in\mathcal{U}_s$
	}
}
\Return{ $(1, \{\bm{\pi}^0\}\cup \{\bm{\pi}^{C}\}_{C\in\mathcal{C}^{\I}(M^{\I})})$}
}	
	\caption{APD for binary MMDPs}\label{alg:APD}
}
\end{algorithm}

In Algorithm~\ref{alg:APD}, an informative MDP $M^{\I}$ in line~\ref{alg:IMDP} can be constructed by traversing the states and actions in the original MDPs. We can then compute of the set of MECs in line~\ref{alg:mec} via~\cite[Algorithm 47]{CB-JK:08}; the set $\mathcal{C}^{\I}(M^{\I})$ of informative MECs in line~\ref{alg:inmec} is defined by
\begin{equation}\label{eq:informativeMEC}
\mathcal{C}^{\I}(M^{\I})=\{(\mathcal{X},\mathcal{U})\in\mathcal{C}(M^{\I})\,|\,\exists(s,a)\in\ISA^{\p}, s\in \mathcal{X},a\in \mathcal{U}\},
\end{equation}
which consists of all MECs that contain at least one informative state-action pair. Note that $\mathcal{C}^{\I}(M^{\I})$ is always nonempty since it contains the MECs $(\bot_i,a^{\bot_i})$ for $i\in\{1,2\}$; we compute the set of states $\mathcal{R}^{\textup{max}}\subset\mathcal{S}^{\p}$ that have a maximum probability of one to reach the states in $\mathcal{C}^{\I}(M^{\I})$ via the graph-theoretic algorithm \cite[Algorithm 45]{CB-JK:08}~in line~\ref{alg:setreachmax}. In line~\ref{alg:reachpolicy}, the policy $\bm{\pi}^0$  at a state $s\in\mathcal{R}^{\textup{max}}\setminus\mathcal{C}^{\I}(M^{\I})$ takes any action $a\in\mathcal{A}^{\p}_{s}$ that satisfies $\Supp(\delta^{\I}(\cdot\,|\,s,a))\subset\mathcal{R}^{\textup{max}}$ with probability (w.p.) $1$. Such an action always exists for the states in $\mathcal{R}^{\textup{max}}\setminus\mathcal{C}^{\I}(M^{\I})$ according to the form of the Bellman optimality equation \cite[Theorem 10.100]{CB-JK:08}.

%
%

Theorem~\ref{thm:alg} guarantees the correctness of Algorithm~\ref{alg:APD}. 
\begin{thom}[Correctness of Algorithm~\ref{alg:APD}]\label{thm:alg}
Given a binary MMDP $\mathcal{M}=\{M_1,M_2\}$, Algorithm~\ref{alg:APD} determines in finite time the existence of a policy that achieves APD for $\mathcal{M}$ and synthesizes a policy when one exists. 
\end{thom}
\begin{proof}
We postpone the proof to Appendix~\ref{sec:proofalg}.
\end{proof}

A few remarks on Algorithm~\ref{alg:APD} are in order.

\begin{rek}[Polynomial time complexity]
Algorithm~\ref{alg:APD} has polynomial time complexity in the total number of states $|\mathcal{S}^{\p}|$ and the total number of actions $|\mathcal{A}^{\p}|$. Specifically, the construction of an informative MDP in line~\ref{alg:IMDP} takes $\mathcal{O}(|\mathcal{A}^{\p}||\mathcal{S}^{\p}|^2)$; the MEC decomposition in line~\ref{alg:mec} takes $\mathcal{O}(|\mathcal{A}^{\p}||\mathcal{S}^{\p}|^3)$ 
\cite[Page 879]{CB-JK:08}; computing the set $\mathcal{R}^{\textup{max}}$ in line~\ref{alg:setreachmax} takes $\mathcal{O}(|\mathcal{A}^{\p}||\mathcal{S}^{\p}|^3)$\cite[Page 860]{CB-JK:08}.
\end{rek}

\begin{rek}[Pure dependence on the structure of informative MDPs]\label{rek:dependencestructure}
	The outcome of Algorithm~\ref{alg:APD} depends purely on the structure of informative MDPs instead of the exact transition probabilities. Therefore, the selection of $\gamma\in(0,1)$ in Definition~\ref{def:informative_MDP} can be arbitrary. 
\end{rek}

\begin{rek}[APD from any state]
The existence of a policy that achieves APD for a binary MMDP depends crucially on the initial state, as demonstrated in line~\ref{alg:maximumreach} of Algorithm~\ref{alg:APD}. The set $\mathcal{R}^{\textup{max}}$ contains exactly those states from which there exists a policy such that APD can be achieved. Moreover, the policies $\bm{\pi}^0$ and $\bm{\pi}^C$ stay the same regardless of the initial state. Therefore, Algorithm~\ref{alg:APD}, subject to minor modifications, is able to determine APD from all states in one shot. On the other hand, when the initial condition is a distribution over the state space, APD can be determined by examining if the support of the initial distribution is a subset of $\mathcal{R}^{\textup{max}}$.
\end{rek}

We next show that for a given binary MMDP $\mathcal{M}$, if there exists a policy $\bm{\pi}$ under which APD is achieved, then the BC  $B(t,\bm{\pi})$ under $\bm{\pi}$ converges to zero exponentially fast with the length $t$ of the history.

\begin{lema}[Exponential convergence of the BC]\label{lema:expconvergence}
Given a binary MMDP $\mathcal{M}=\{M_1,M_2\}$, suppose APD for $\mathcal{M}$ is achieved under a stationary policy $\bm{\pi}$. Then, the BC converges exponentially fast, i.e., there exist $c>0$ and $0<\lambda<1$ such that
\begin{equation*}
B(t,\bm{\pi})\leq c\lambda^t.
\end{equation*}
\end{lema}
\begin{proof}
We postpone the proof to Appendix~\ref{sec:proofexpconvergence}.
\end{proof}

Lemma~\ref{lema:expconvergence} shows that the BC decays exponentially fast when we apply the policy synthesized by Algorithm~\ref{alg:APD}. 
Note that when the estimated prior $q$ is accurate and close to the true prior $\theta$ in the Bayesian rule, the BC serves as a tight bound for the error probability. In this case, we can confidently decide between candidate hypotheses based on a potentially short observation sequence due to the rapid decay of the BC.

\section{Detection of general MMDPs}\label{sec:general}
In this section, we study the detection problem for an MMDP $\mathcal{M}=\{M_i\}_{i\in\mathbb{N}_{1}^N}$ that consists of $N\geq3$ MDPs.  We are interested in asymptotically perfectly detecting any MDP in $\mathcal{M}$ that could govern the underlying transition dynamics.

\subsection{APD for  multiple hypotheses}
Similar to the binary case, we first derive bounds on the probability of error for the MAP rule when there are more than two hypotheses. 

\begin{lema}[Probability of error for multiple hypotheses]\label{lema:BCgeneral}
	Given $N$ hypotheses and for  $i\in\mathbb{N}_{1}^N$, let $f_i: \mathbb{R}^n\rightarrow[0,\infty)$,  $q_i>0$ and $\theta_i>0$ be the PDF, the estimated prior and the true prior of the $i$-th hypothesis, respectively. Then the probability of error $P_{\textup{error}}$ of the MAP rule satisfies 
	\begin{multline}\label{eq:boundsonBgeneral}
	\frac{1}{2}\max_{k\in\mathbb{N}_{1}^N}\{\sum_{i\neq k}\min\{\theta_i,\theta_k\}B_{ik}^2\}\leq P_{\textup{error}}\\
	\leq\max_i\{\frac{\theta_i}{q_i}\}\cdot\sum_{i<j}\sqrt{q_iq_j}B_{ij},
\end{multline}
where $B_{ij}$ is the BC between $f_i(z)$ and $f_j(z)$ defined by
\begin{equation*}
B_{ij}=\int_{\mathbb{R}^n}\sqrt{f_i(z)f_j(z)}dz.
\end{equation*}
\end{lema}
\begin{proof}
	We postpone the proof to Appendix~\ref{appendix:BCgeneral}.
\end{proof}

\begin{rek}[Loose upper bound]
The upper bound on the probability of error in~\eqref{eq:boundsonBgeneral} is loose in the sense that the summation might be greater than $1$. However, it becomes effective when $B_{ij}$'s are sufficiently small for any distinct pair of $i,j\in\mathbb{N}_{1}^N$, which is exactly what we aim for.
\end{rek}

The bounds on the probability of error in~\eqref{eq:boundsonBgeneral} reduce to those in~\eqref{eq:boundsonB} when $N=2$. Moreover, we observe that APD is achieved if and only if for every pair of $i,j\in\mathbb{N}_{1}^N$, we have the BC $B_{ij}=0$.

To formulate the APD problem for a general MMDP $\mathcal{M}$, we define the BC for each pair of MDPs in $\mathcal{M}$ under a policy $\bm{\pi}$, i.e., for $i,j\in\mathbb{N}_{1}^N$,
\begin{equation*}
B_{ij}(t,\bm{\pi})=\sum_{h_t\in \mathcal{H}_t}\sqrt{\mathbb{P}_i^{\bm{\pi}}(h_t)\mathbb{P}_j^{\bm{\pi}}(h_t)},
\end{equation*}
where $\mathcal{H}_t$ is the union set of histories in all MDPs in $\mathcal{M}$. In order to achieve APD for $\mathcal{M}$, we need to design a policy $\bm{\pi}$ such that 
the BCs satisfy $B_{ij}(\bm{\pi})=\lim_{t\rightarrow\infty}B_{ij}(t,\bm{\pi})=0$ for all $i,j\in\mathbb{N}_1^N$. In other words, the policy $\bm{\pi}$ must allow us to distinguish all pairs of MDPs in $\mathcal{M}$ simultaneously. 

We emphasize that finding a policy for each pair of MDPs in $\mathcal{M}$ separately need not work in general as these found policies may not be consistent with each other. Nevertheless, a necessary condition for APD for $\mathcal{M}$ is that, at least for every pair of MDPs, there exists one policy that achieves APD for this pair. Our solution method deals with $N$ MDPs altogether. 

\subsection{Base case: no identity-revealing transitions}\label{sec:basecase}
Before solving the general problem, we discuss a special case that can be addressed by applying a slightly modified Algorithm~\ref{alg:APD}.
Specifically, we consider an MMDP $\mathcal{M}$ where all MDPs in $\mathcal{M}$ have exactly the same transition structure and there are no identity-revealing transitions, i.e., for all $s,s'\in\mathcal{S}$ and $a\in\mathcal{A}_s$, either $\delta_i(s'\,|\,s,a)>0$ for all $i\in\mathbb{N}_{1}^N$ or $\delta_i(s'\,|\,s,a)=0$ for all $i\in\mathbb{N}_{1}^N$.

When all the MDPs in $\mathcal{M}$ have the same transition structure, the informative MDPs for any pair of MDPs in $\mathcal{M}$ also have the same structure. Moreover, the structure of these pairwise informative MDPs is the same as that of the MDPs themselves (except that the informative MDPs have two additional non-reachable terminal states). Therefore, we can use any of the MDPs in $\mathcal{M}$ in line~\ref{alg:IMDP} of Algorithm~\ref{alg:APD}. Based on the definition of the informative state-action pairs in Definition~\ref{def:informative_state_action}, we introduce the set of informative state-action pairs for every pair of MDPs $M_i, M_j\in\mathcal{M}$ as
\begin{multline*}
    \ISA_{ij}=\{(s,a)\,|\,\\s\in\mathcal{S},a\in\mathcal{A}_{s}, (s,a)\textup{ is informative in } \{M_i, M_j\}\}.
\end{multline*}
Then, we modify the definition of the set of informative MECs $\mathcal{C}^{\I}(M^{\I})$ in~\eqref{eq:informativeMEC} to be 
 \begin{multline}\label{eq:informativeMECbase}
	\mathcal{C}^{\I}(M^{\I})=\{(\mathcal{X},\mathcal{U})\in\mathcal{C}(M^{\I})\,|\, \forall i,j\in\mathbb{N}_{1}^N,\\ \exists s\in \mathcal{X},a\in \mathcal{U}_s, \textup{such that } (s,a)\in\ISA_{ij}\}.
\end{multline}
In~\eqref{eq:informativeMECbase}, we require the informative MECs in  $\mathcal{C}^{\I}(M^{\I})$ to contain at least one informative state-action pair from each set $\ISA_{ij}$.
%

The following lemma guarantees that with the modified definition for the informative MECs in~\eqref{eq:informativeMECbase}, Algorithm~\ref{alg:APD} solves the APD problem for an MMDP $\mathcal{M}$ when all MDPs in $\mathcal{M}$ have the same transition structure. 

\begin{lema}[Base case for APD for general MMDPs]\label{lema:basecase}
Given an MMDP $\mathcal{M}=\{M_i\}_{i\in\mathbb{N}_{1}^N}$ where all MDPs in $\mathcal{M}$ have the same transition structure, then Algorithm~\ref{alg:APD} with the modified definition for informative MECs~\eqref{eq:informativeMECbase}, determines in finite time the existence of a policy that achieves APD for $\mathcal{M}$ and synthesizes a policy when one exists.
\end{lema}
\begin{proof}
When all MDPs in $\mathcal{M}$ have the same transition structure, the informative MDPs for any pair of MDPs in $\mathcal{M}$ have the same structure and are the same across all pairs. We therefore only need to focus on one informative MDP of any pair of MDPs. Moreover, we also note that there are no transitions leading to the terminal states $\bot_1$ and $\bot_2$ in this informative MDP.
	
For any pair of MDPs $M_i, M_j\in\mathcal{M}$, by the proof of Theorem~\ref{thm:alg}, the BC $B_{ij}(\bm{\pi})=0$ if and only if the probability of reaching the set of informative MECs that contain at least one pair of informative state-action pair in $\ISA_{ij}$ is one. To achieve APD for $\mathcal{M}$, by~Lemma~\ref{lema:BCgeneral}, it is necessary and sufficient that for all $i,j\in\mathbb{N}_{1}^N$, we have $B_{ij}(\bm{\pi})=0$. Therefore, APD is achieved for $\mathcal{M}$ if and only if the probability of reaching the set of informative MECs that contain at least one pair of informative state-action pair in $\ISA_{ij}$ from each pair of $i,j\in\mathbb{N}_{1}^N$ is one. Moreover, the policy $\bm{\pi}^C$ visits all state-action pairs inside an MEC infinitely often. We therefore conclude the correctness of Algorithm~\ref{alg:APD}.
\end{proof}

\subsection{APD for general MMDPs}
Our construction of the informative MDP in the binary case exploits the fact that the identity of the underlying MDP is revealed immediately when an identity-revealing transition occurs. Therefore, we could introduce terminal states and direct all the identity-revealing transitions to them in the respective MDPs. However, for general MMDPs, we may not terminate the detection process even when identity-revealing transitions occur because those transitions may still be possible in multiple remaining MDPs. To address this issue, instead of introducing terminal states for the identity-revealing transitions, we solve the APD problem for a new MMDP after each identity-revealing transition. For instance, starting from the initial state $s_{\init}$, if the transitions to a state $s$ after taking an action $a\in\mathcal{A}_{s_{\init}}$ in $\mathcal{M}$ satisfy $\delta_{i}(s\,|\,s_{\init},a)=0$ for all $i\in \mathcal{N}_0$ and $\delta_{i}(s\,|\,s_{\init},a)>0$ for all $i\in \mathcal{N}_1$ with $\mathcal{N}_0\cup \mathcal{N}_1=\mathbb{N}_{1}^N$ and $\mathcal{N}_0\cap \mathcal{N}_1=\emptyset$, then after observing the transition $(s_{\init},a,s)$, we only need to focus on the MMDP $\mathcal{M}'=\{M_i\}_{i\in \mathcal{N}_1}$ with the initial state $s$. Note that $\mathcal{M}'$ is just another MMDP, and we could solve it if we had an algorithm for APD for general MMDPs. This observation suggests a recursive structure to our algorithm. Moreover, if we were able to determine whether there exists a policy that achieves APD for $\mathcal{M}'$ and compute it when one exists, we could modify the transition $(s_{\init},a,s)$ to $(s_{\init},a,\bot^{\g}_1)$ or $(s_{\init},a,\bot^{\g}_0)$ where $\bot^{\g}_1$ is a ``good'' terminal state indicating that it is possible to asymptotically perfectly detect the remaining MDPs in $\mathcal{M}'$ starting from $s$ and $\bot^{\g}_0$ is a ``bad'' one indicating the opposite.

The idea outlined above is the key to systematically addressing the identity-revealing transitions. The algorithm calls itself to solve APD problems for MMDPs that consist of fewer MDPs than the original MMDP. There are two base cases for the recursive part of the algorithm:  i) the binary MMDPs and ii) the case discussed in Section~\ref{sec:basecase}. We use a transition system to store the available transitions, whose definition is given below.

\begin{defi}[Transition systems]
A transition system $T$ is a tuple $T=(\mathcal{Y},\mathcal{B},\mathcal{T},y_{\init})$ where
\begin{enumerate}
\item $\mathcal{Y}$ is a finite set of states;
\item $\mathcal{B}=\cup_{y\in\mathcal{Y}}\mathcal{B}_y$ is the union of the finite sets of actions $\mathcal{B}_y$ available at the state $y\in\mathcal{Y}$; 
\item $\mathcal{T}\subset\mathcal{Y}\times \mathcal{B}\times \mathcal{Y}$ is a set of possible transitions;
\item $y_{\init}\in\mathcal{Y}$ is the initial state.
\end{enumerate}
\end{defi}

Transition systems are closely related to MDPs. For a given MDP, we can construct the underlying transition system by storing the transitions that have positive probabilities in the MDP. On the other hand, for a given transition system, we can define a set of MDPs compatible with it by assigning positive transition probabilities to the transitions. Moreover, since the concept of (maximal) end components for MDPs in Definition~\ref{def:MEC} depends solely on the transition structure of the MDPs, it carries over directly to transition systems. We also note that, by the discussion in Remark~\ref{rek:dependencestructure}, for a binary MMDP, the existence and synthesis of a policy that achieves APD can be completely determined by looking at the associated transition system of the informative MDP.

\SetKwInput{kwInit}{Init}
\begin{algorithm}[hbt!]
	\KwIn{An MMDP $\mathcal{M}=\{M_i\}_{i\in \mathcal{N}}$, an initial state $s_{\init}$ and a policy set $\bm{\Pi}$}
	\KwOut{A boolean variable indicating whether a policy for APD exists and a policy set $\bm{\Pi}$}
	\kwInit{Empty queues $Q_1$ and $Q_2$, a transition system $T=(\mathcal{Y},\mathcal{B},\mathcal{T})$ where $\mathcal{Y}=\{\bot^{\g}_0,\bot^{\g}_1\}$, $\mathcal{B}=\mathcal{B}_{\bot^{\g}_0}\cup\mathcal{B}_{\bot^{\g}_1}$ with $\mathcal{B}_{\bot^{\g}_0}=\{a^{\bot^{\g}_0}\}$ and $\mathcal{B}_{\bot^{\g}_1}=\{a^{\bot^{\g}_1}\}$, and $\mathcal{T}=\{(\bot^{\g}_i,a^{\bot^{\g}_i},\bot^{\g}_i)\}_{i\in\{0,1\}}$ }
	\SetKwBlock{Begin}{function}{end function}
	\Begin($\texttt{APD} {(} \{M_i\}_{i\in \mathcal{N}}, s_{\init}, \bm{\Pi} {)}$)
	{	
		\uIf{$|\mathcal{N}|==2$\label{alg2:binary1}}
		{
			$(FLAG,\bm{\Pi}_0)\gets \texttt{BiAPD}(\{M_i\}_{i\in \mathcal{N}},s_{\init})$\label{alg2:binary2}\\
			\Return{$(FLAG,\bm{\Pi}\cup\bm{\Pi}_0)$	\label{alg2:binary3}}		
		}
		\texttt{Insert}($Q_1, s_{\init}$)\label{alg2:inqueque}, label $s_{\init}$ as explored\\		
		\While{$Q_1$ is not empty\label{alg2:bfsstarts}}
		{
			$s\gets$\texttt{Retrieve}($Q_1$)
			$\mathcal{Y}\gets\mathcal{Y}\cup\{s\}$, $\mathcal{B}_s\gets\emptyset$\label{alg2:explorestart}\\
			\For{$a\in\mathcal{A}_s$}
			{
				$\mathcal{B}_s\gets\mathcal{B}_s\cup\{a\}$\\
				\For{$s'\in\mathcal{S}$}
				{
					$\mathcal{N}'\gets\{i\in\mathcal{N}\,|\,\delta_{i}(s'\,|\,s,a)>0\}$\\
					\uIf{$|\mathcal{N}'|>0$}
					{
						\uIf{$|\mathcal{N}'|==1$\label{alg2:singleton0}}
						{
							$\mathcal{T}\gets\mathcal{T}\cup\{(s,a,\bot^{\g}_1)\}$\label{alg2:singleton1}
						}
						\uElseIf{$|\mathcal{N}'|==2$\label{alg2:binaryN0}}
						{
							 $(FLAG,\bm{\Pi}_0)\gets \texttt{BiAPD}(\{M_i\}_{i\in\mathcal{N}'},s')$\\
							 $\bm{\Pi}\gets\bm{\Pi}\cup\bm{\Pi}_0$\\							 
							$\mathcal{T}\gets\mathcal{T}\cup\{(s,a,\bot^{\g}_{FLAG})\}$\label{alg2:binaryN1}
						
						}
						\uElseIf{$\mathcal{N}'==\mathcal{N}$\label{alg2:norevealing0}}
						{
							$\mathcal{T}\gets\mathcal{T}\cup\{(s,a,s')\}$\\
							\uIf{$s'$ is not explored}
							{
								\texttt{Insert}($Q_1, s'$)\\ label $s'$ as explored\label{alg2:norevealing1}
							}
						}
						\Else
						{							
							\texttt{Insert}($Q_2, (\mathcal{N}',(s,a,s'))$)\label{alg2:exploreend}
						}
					}
				}		
			}
		}
		\While{$Q_2$ is not empty\label{alg2:recursion0}}
		{
			$(\mathcal{N}',(s,a,s'))\gets$\texttt{Retrieve}($Q_2$)\\
			$(FLAG,\bm{\Pi})=\texttt{APD}(\{M_i\}_{i\in\mathcal{N}'},s',\bm{\Pi})$\label{alg2:recursivecall}\\
			$\mathcal{T}\gets\mathcal{T}\cup\{(s,a,\bot^{\g}_{FLAG})\}$\label{alg2:recursion1}

		}
	 	Find MECs $\mathcal{C}(T)$ and  informative MECs $\mathcal{C}^{\I}(T)$\label{alg2:transitionsys0}\\
	 	Compute the set of states that reach $\mathcal{C}^{\I}(T)$ w.p. 1: $\mathcal{R}^{\textup{max}}=\{s\in\mathcal{Y}\,|\,\mathbb{P}_{s,T}^{\textup{max}}(\textup{reach}(\mathcal{C}^{\I}(T)))=1\}$\label{alg2:Rmax}\\
	 	\uIf{$s_{\init}\notin\mathcal{R}^{\textup{max}}$}
	 	{
	 		\Return{$(0, \emptyset)$}
	 	}	
 	\Else
 {
 	 	Synthesize $\bm{\pi}_{\mathcal{N}}^0$ that reaches $\mathcal{C}^{\I}(T)$ w.p. $1$\label{alg2:io0}\\
 \For{$C=(\mathcal{X},\mathcal{U})\in\mathcal{C}^{\I}(T)$}
 {
 	\For{$s\in\mathcal{X}$}
 	{
 		$\bm{\pi}_{\mathcal{N}}^C(a\,|\,s)=\frac{1}{|\mathcal{U}_s|}$ for $a\in\mathcal{U}_s$
 	}
 }
 \Return{($1$, $\bm{\Pi}\cup\{\bm{\pi}_{\mathcal{N}}^0\}\cup\{\bm{\pi}_{\mathcal{N}}^C\}_{C\in\mathcal{C}^{\I}(T)}$)\label{alg2:io1}}
}
	}
	\caption{APD for general MMDPs}\label{alg:APDgeneral}
\end{algorithm}

We present the complete algorithm in Algorithm~\ref{alg:APDgeneral}. Algorithm~\ref{alg:APDgeneral} features two main algorithmic components: the breadth-first search (BFS) and recursion. During the algorithm, we build a transition system $T=(\mathcal{Y},\mathcal{B},\mathcal{T})$ that serves the role of the informative MDP for binary MMDPs, where $\mathcal{Y}$ is the state space, $\mathcal{B}=\cup_{y\in\mathcal{Y}}\mathcal{B}_y$ is the union set of actions, and $\mathcal{T}=\{(y,a,y')\,|\,y,y'\in\mathcal{Y},a\in\mathcal{B}_y\}$ is a set of allowable transitions. 
Then, the existence of a policy that achieves APD can be determined by analyzing the transition structure of $T$. The procedure is similar to and consistent with that for the binary case. In fact, since for an MDP $M=(\mathcal{S},\mathcal{A},\delta,s_{\init})$, there exists a unique transition system $T=(\mathcal{S},\mathcal{A},\mathcal{T})$ associated with $M$, where $\mathcal{T}=\{(s,a,s')\,|\,s,s'\in\mathcal{S},a\in\mathcal{A}_s,\delta(s'\,|\,s,a)>0\}$, the informative MDP $M^{\I}$ in line~\ref{alg:IMDP} of Algorithm~\ref{alg:APD} can be replaced by its associated transition system as hinted by the discussions in Remark~\ref{rek:dependencestructure}.

 The detailed workflow of Algorithm~\ref{alg:APDgeneral} is as follows. The algorithm first decides whether the input is a binary MMDP and calls Algorithm~\ref{alg:APD} if it is (lines~\ref{alg2:binary1}-\ref{alg2:binary3}). Otherwise, the initial state $s_{\init}$ enters the queue $Q_1$ and the BFS begins (lines~\ref{alg2:inqueque}-\ref{alg2:bfsstarts}). We explore all the actions and the consequent transitions available at the state $s$ popped out from $Q_1$ (lines~\ref{alg2:explorestart}-\ref{alg2:exploreend}). There are a few possibilities.
 \begin{enumerate}[label=(\roman*)]
\item The transition $(s,a,s')$ is only available in exactly one MDP (lines~\ref{alg2:singleton0}-\ref{alg2:singleton1}), in which case we add a transition $(s,a,\bot^{\g}_1)$ to the transition system indicating that if such a transition occurs, APD is achieved; 
\item The transition $(s,a,s')$ is available in two MDPs (lines~\ref{alg2:binaryN0}-\ref{alg2:binaryN1}), in which case we call Algorithm~\ref{alg:APD} to decide whether a policy exists for the corresponding binary MMDP and add a transition $(s,a,\bot^{\g}_1)$ or $(s,a,\bot^{\g}_0)$ to $\mathcal{T}$ depending on the outcome of the binary algorithm;
\item The transition $(s,a,s')$ is available in all MDPs (lines~\ref{alg2:norevealing0}-\ref{alg2:norevealing1}), in which case the state $s'$ enters $Q_1$ and needs to be further explored;
\item\label{itm:recursive} The transition $(s,a,s')$ is available in more than two but not all MDPs (line~\ref{alg2:exploreend}), in which case we store the possible MDPs and the current transition $(s,a,s')$ in $Q_2$.
 \end{enumerate}
By the end of the BFS phase, the state space $\mathcal{Y}$ of the transition system consists of two terminal states $\bot^{\g}_0$ and $\bot^{\g}_1$, and states in $\mathcal{S}$ that we can reach from $s_{\init}$ in all MDPs in $\mathcal{M}$ following the same history. To deal with case~\ref{itm:recursive} encountered during the BFS, we call \texttt{APD} recursively (lines~\ref{alg2:recursion0}-\ref{alg2:recursion1}). Depending on the returns of the recursive calls, we further update the transition system. After obtaining the returns from the recursive calls, we find the MECs and informative MECs of the transition system $T$ (line~\ref{alg2:transitionsys0}), where the informative MECs consists of those defined in~\eqref{eq:informativeMECbase} for the states in $\mathcal{Y}$ and the singleton $(\bot^{\g}_1,a^{\bot^{\g}_1})$. Finally, we decide if it is possible to visit the informative state-action pairs infinitely often and find the corresponding policies when they exist (lines~\ref{alg2:io0}-\ref{alg2:io1}). 

Theorem~\ref{thm:alggeneral} guarantees the correctness of Algorithm~\ref{alg:APDgeneral}.

\begin{thom}[Correctness of Algorithm~\ref{alg:APDgeneral}]\label{thm:alggeneral}
	Given an MMDP $\mathcal{M}=\{M_i\}_{i\in\mathbb{N}_{1}^N}$, Algorithm~\ref{alg:APDgeneral} determines in finite time the existence of a policy that achieves APD for $\mathcal{M}$ and synthesizes a policy when one exists. 
\end{thom}
\begin{proof}
We postpone the proof to Appendix~\ref{sec:proofalggeneral}.
\end{proof}

A few remarks on Algorithm~\ref{alg:APDgeneral} are in order.

\begin{rek}[Policies with memory]\label{rek:memorypolicy}
A distinct feature of the policies that achieve APD for general MMDPs, if they exist, is that they have memory. In particular, the policies depend on the current state and the current set of MDPs that are ``active''. On the other hand, if we augment the state variable in $\mathcal{S}$ with subsets of $\mathbb{N}_{1}^N$, then we obtain memoryless policies.
\end{rek}

\begin{rek}[Time complexity and improvement]
Given an input MMDP $\{M_i\}_{i\in\mathbb{N}_1^N}$, we may need to examine almost all proper subsets of $\mathbb{N}_1^N$ through recursive calls. Therefore,  Algorithm~\ref{alg:APDgeneral} runs with time complexity that is  exponential with the number of MDPs in the input MMDP. To avoid repeated recursive calls with the same input, we can keep track of all recursive calls and retrieve the results directly before executing line~\ref{alg2:recursivecall} (memoization).
\end{rek}


\begin{rek}[Special cases]
There are two interesting special cases that can be handled by Algorithm~\ref{alg:APDgeneral}: i) detecting a specific MDP in $\mathcal{M}$; ii) the set of MDPs in $\mathcal{M}$ is divided into two groups and detecting which group contains the ground truth MDP. To address these two cases, one only needs to modify the BFS part (deciding the transitions to $\bot_1^{\g}$ and $\bot_0^{\g}$) and the definition of informative MECs.
\end{rek}

We finally note that due to the policies' dependence on the current set of active MDPs as discussed in Remark~\ref{rek:memorypolicy}, the computation of the BC $B_{ij}(t,\bm{\pi})$ for $M_i,M_j\in\{M_i\}_{i\in\mathcal{N}}$ and the policy $\bm{\pi}$ returned by Algorithm~\ref{alg:APDgeneral} that achieves APD, is slightly more involved than the binary case. Nevertheless, since the policies become memoryless when considering the augmented state space, we can compute the BC similarly.

\section{Numerical examples}\label{sec:numerical}
We demonstrate the effectiveness of our algorithms through two numerical examples, i.e., intruder detection in urban environments and an MDP-based recommendation system.

\subsection{Bayesian belief updates}
Given an MMDP $\mathcal{M}=\{M_i\}_{i\in\mathbb{N}_{1}^N}$ with the initial state $s_{\init}$ and the estimated prior probabilities $q_i$ of each MDP $M_i\in\mathcal{M}$, we can calculate the posterior probability for $M_i$ based on the actions taken and observed consequent transitions according to the Bayes' rule. Specifically, let $b(t,s_{t})\in\Delta_{N}$ be the \emph{belief vector} over the set of MDPs in $\mathcal{M}$ at step $t\in\mathbb{N}_{\geq0}$, where $b_i(t,s_t)$ is the posterior probability of $M_i$ at step $t$, then we can recursively update $b(t,s_t)$ as follows,
\begin{equation}\label{eq:bayesianupdate}
b_i(t+1,s_{t+1}) = \frac{b_i(t,s_t)\delta_i({s_{t+1}\,|\,s_t,a_t})}{\sum_{j=1}^Nb_j(t,s_t)\delta_j({s_{t+1}\,|\,s_t,a_t})},
\end{equation}
where $b_i(0,s_{0})=b_i(0,s_{\init})=q_i$ for $i\in\mathbb{N}_{1}^N$. The evolution of the belief vectors in~\eqref{eq:bayesianupdate} depends on the realized histories $(s_0,a_0,s_1,\cdots)$. However, the theories developed in this paper guarantee that under a policy that achieves APD for $\mathcal{M}$, if one exists, the belief vector $b(t,s_t)$ converges to the standard unit vector $\mathbbb{e}_{i'}$ when $M_{i'}\in\mathcal{M}$ is the ground truth MDP that generates the histories.

\subsection{Intruder detection}
Our first example concerns intruder detection in urban environments. We consider an $8\times8$ grid world representing an urban area, as shown in Fig.~\ref{fig:environment}. The human target in the environment can be of two types: a normal person or an intruder. We model the behavior of these two types of agents by two MDPs $M_{\textup{normal}}$ and $M_{\textup{intruder}}$, where the state space of the MDPs consists of possible locations of the agents. The green region in Fig.~\ref{fig:environment} stands for some public facility, e.g., a park, in the environment. Outside the green region, the two types of  agents have the same behavior and gradually move towards the green region randomly. Inside the green region, there are two actions available for monitoring the area:  passive observation and active surveillance, to which the two types of agents respond differently. Specifically, the normal person stays inside the green region with high probability no matter what action is applied. In contrast, the intruder has the same behavior as the normal agent when the passive observation is in effect, but he/she  leaves the region with high probability when the region is under active surveillance. The modeling captures the behavior of an intruder who intends to investigate the green region while avoiding being identified by a surveillance system. After the agents leave the green region, they will reach it again according to their random behavior outside the region.

\begin{figure}[http]
	\centering
	\includegraphics[scale=0.44]{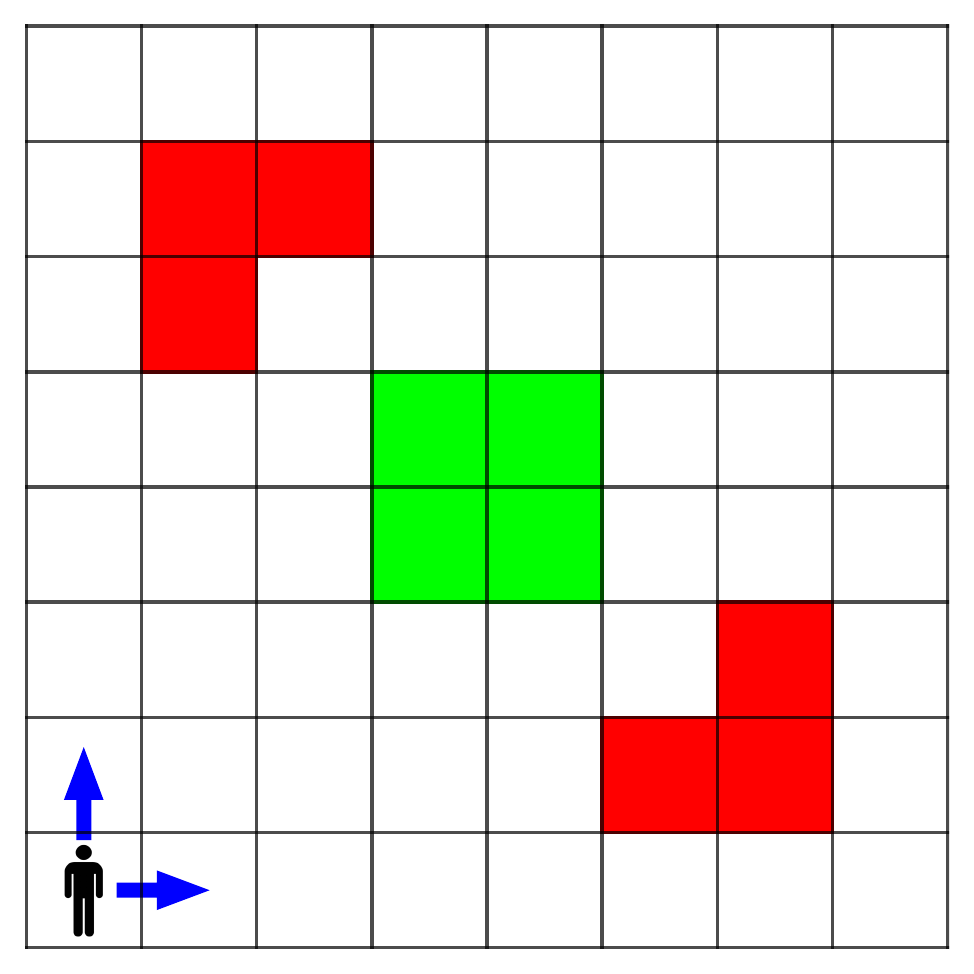}
	\includegraphics[scale=0.44]{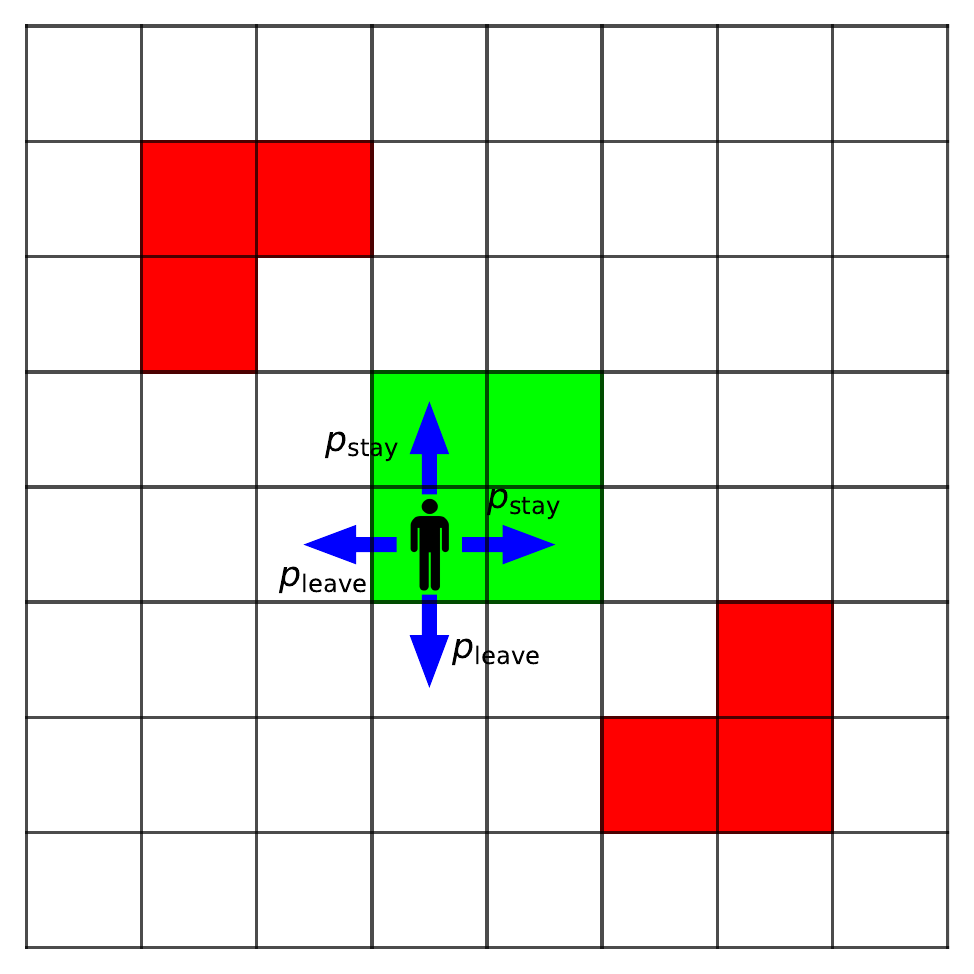}
	\caption{An urban environment with a human target. The person tends to reach the green region and can move to one of the neighboring locations randomly at each step. The red region represents obstacles where the agents cannot move into. Inside the green region, the behaviors of the two types of agents are different.}\label{fig:environment}
\end{figure}

In our experiment, we choose $p_{\rm stay}=0.35$ and $p_{\rm leave}=0.15$ in Fig.~\ref{fig:environment} for the normal person regardless of the actions and the intruder when the passive observation is in effect. Under the active surveillance, the intruder leaves the region with probability $0.35$ and stays with probability $0.15$. We use Algorithm~\ref{alg:APD} to synthesize a policy $\bm{\pi}$ that achieves APD for the binary MMDP $\mathcal{M}=\{M_{\rm normal}, M_{\rm intruder}\}$\footnote{Inside the green region, we adopt a slightly different policy from that given by Algorithm~\ref{alg:APD}. Instead of taking passive observation and active surveillance with equal probability, we always take the surveillance action so as to identify the target more rapidly.}. To simulate the detection process, we uniformly randomly pick one of the MDPs in $\mathcal{M}$, apply the policy $\bm{\pi}$ and update the belief vector according to~\eqref{eq:bayesianupdate}. The evolution of the belief vectors for four realized scenarios are shown in Fig.~\ref{fig:intruder}.

\begin{figure}[http]
	\centering
	\includegraphics[scale=0.29]{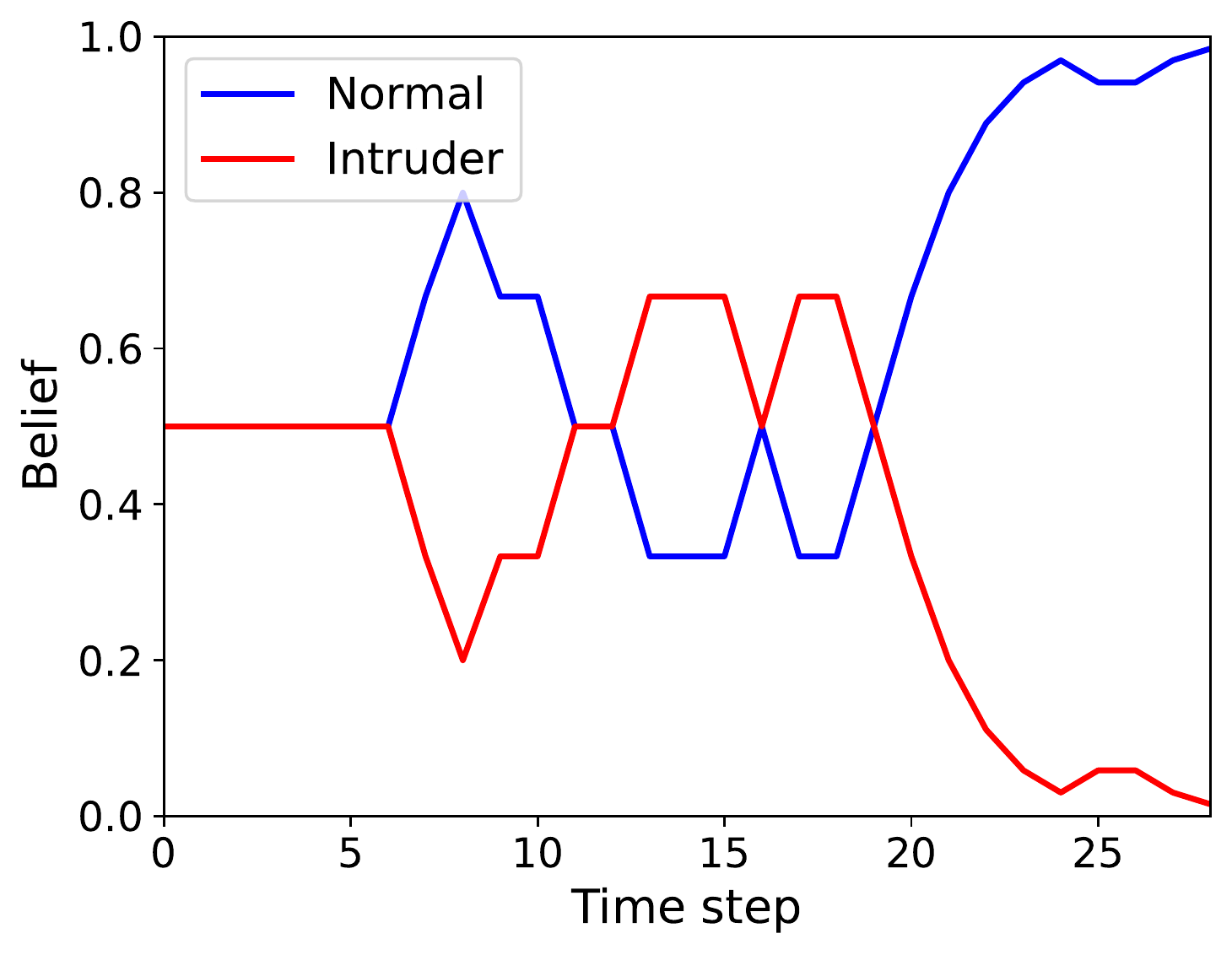}
	\includegraphics[scale=0.29]{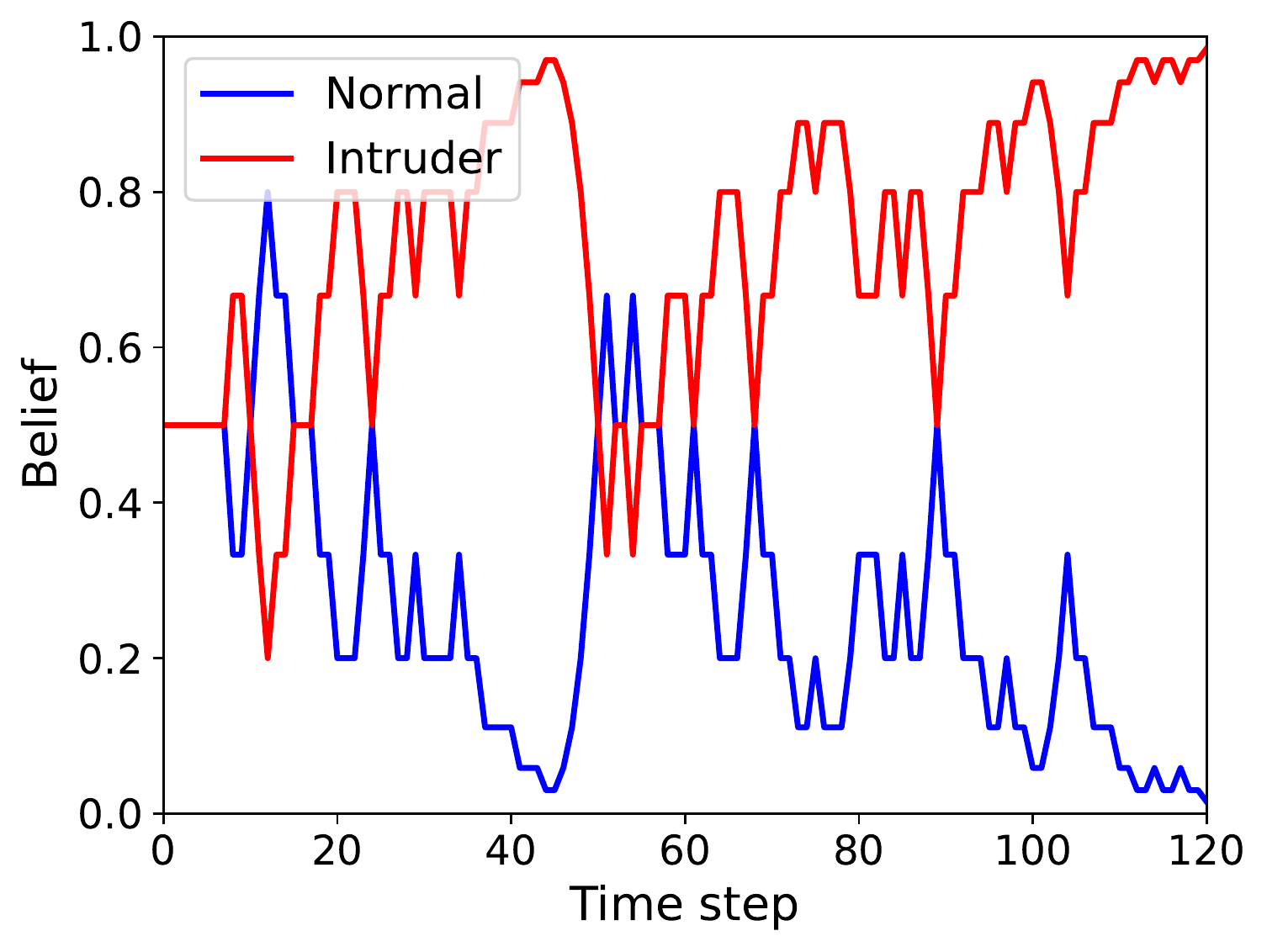}
	\includegraphics[scale=0.29]{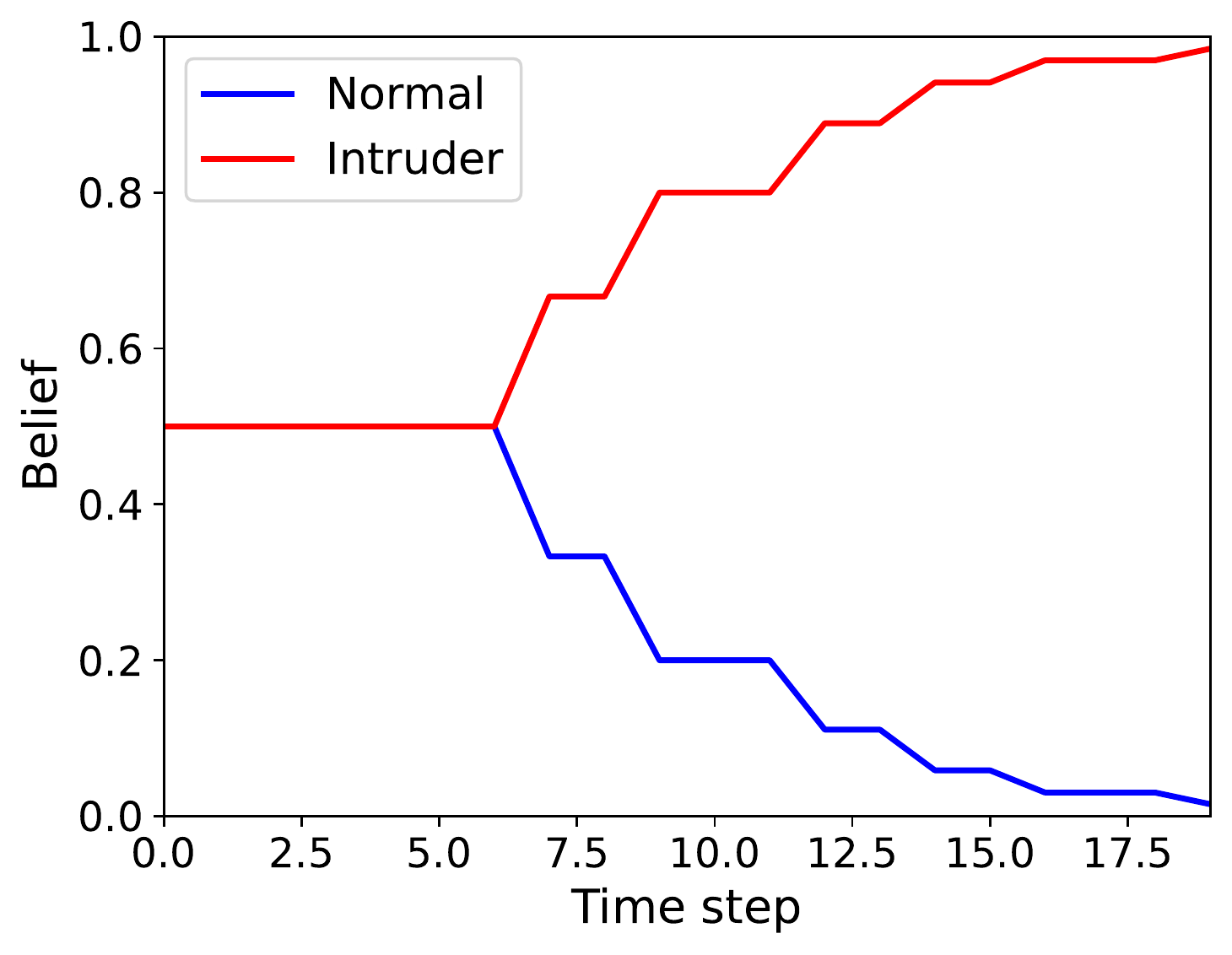}
	\includegraphics[scale=0.29]{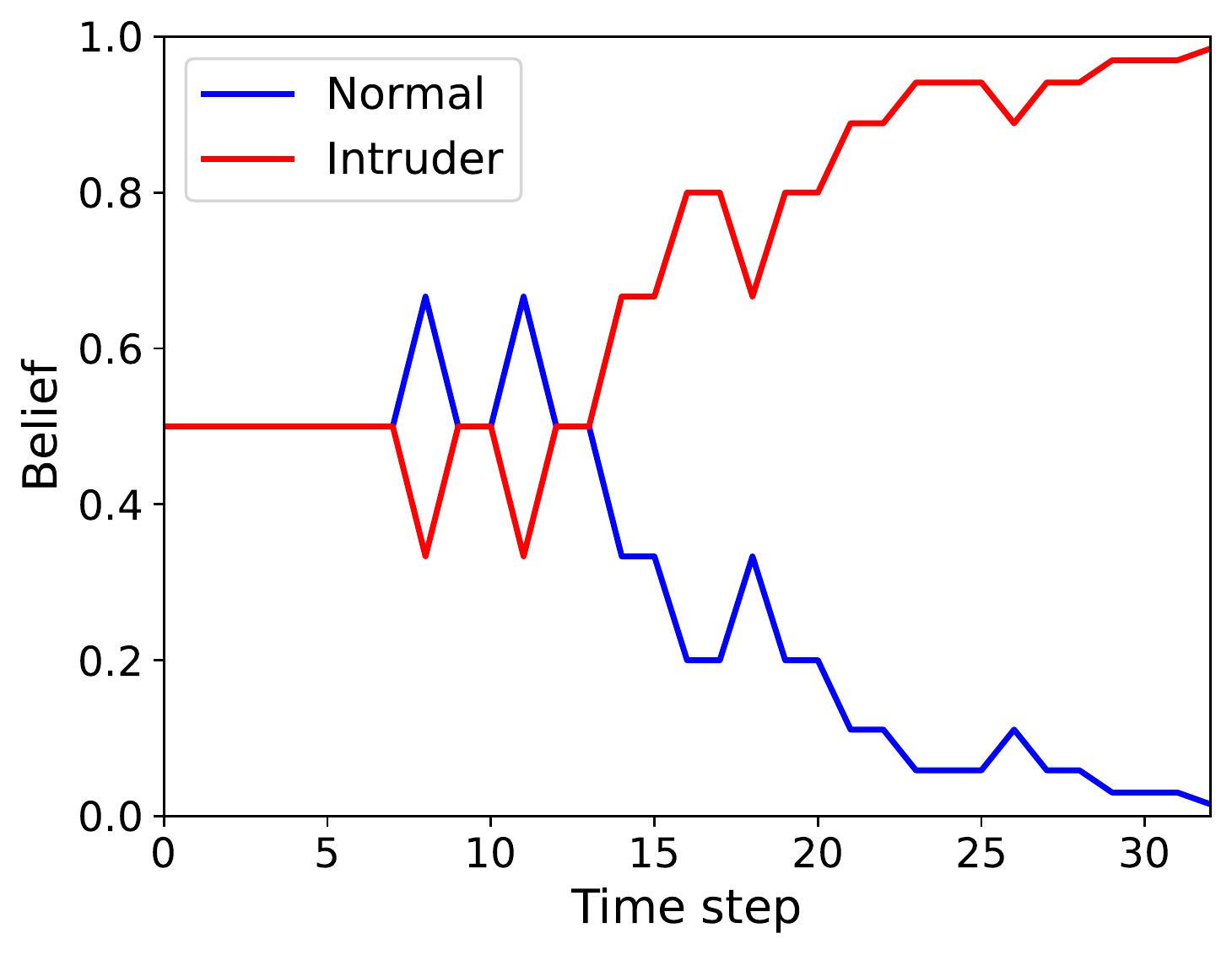}	
	\caption{The evolution of beliefs over agent types in the intruder detection.}\label{fig:intruder}
\end{figure}


In all the scenarios shown in Fig~\ref{fig:intruder}, the agent starts from the bottom left corner of the environment, and the belief update stops when the belief over one of the agent types is greater than $0.98$.  From Fig~\ref{fig:intruder}, we observe that the belief vector eventually correctly indicates the agent type in all cases despite experiencing some transient fluctuations resulting from the intrinsic randomness of the agent's movement and reactions. The belief vector stays constant at the beginning because the two agents behave the same before they reach the green region. We also plot the upper bound in~\eqref{eq:boundsonB} for the probability of error in Fig.~\ref{fig:upperbound} when the estimated and true priors are equal. At the beginning, when the target has not reached the green region, the BC does not change. The plot is consistent with Lemma~\ref{lema:expconvergence}, i.e., the bound decreases exponentially fast with the length of observations. We also note that it is possible to deal with multiple targets in the area simultaneously by independently running one belief vector for each target.

\begin{figure}[http]
	\centering
	\includegraphics[scale=0.44]{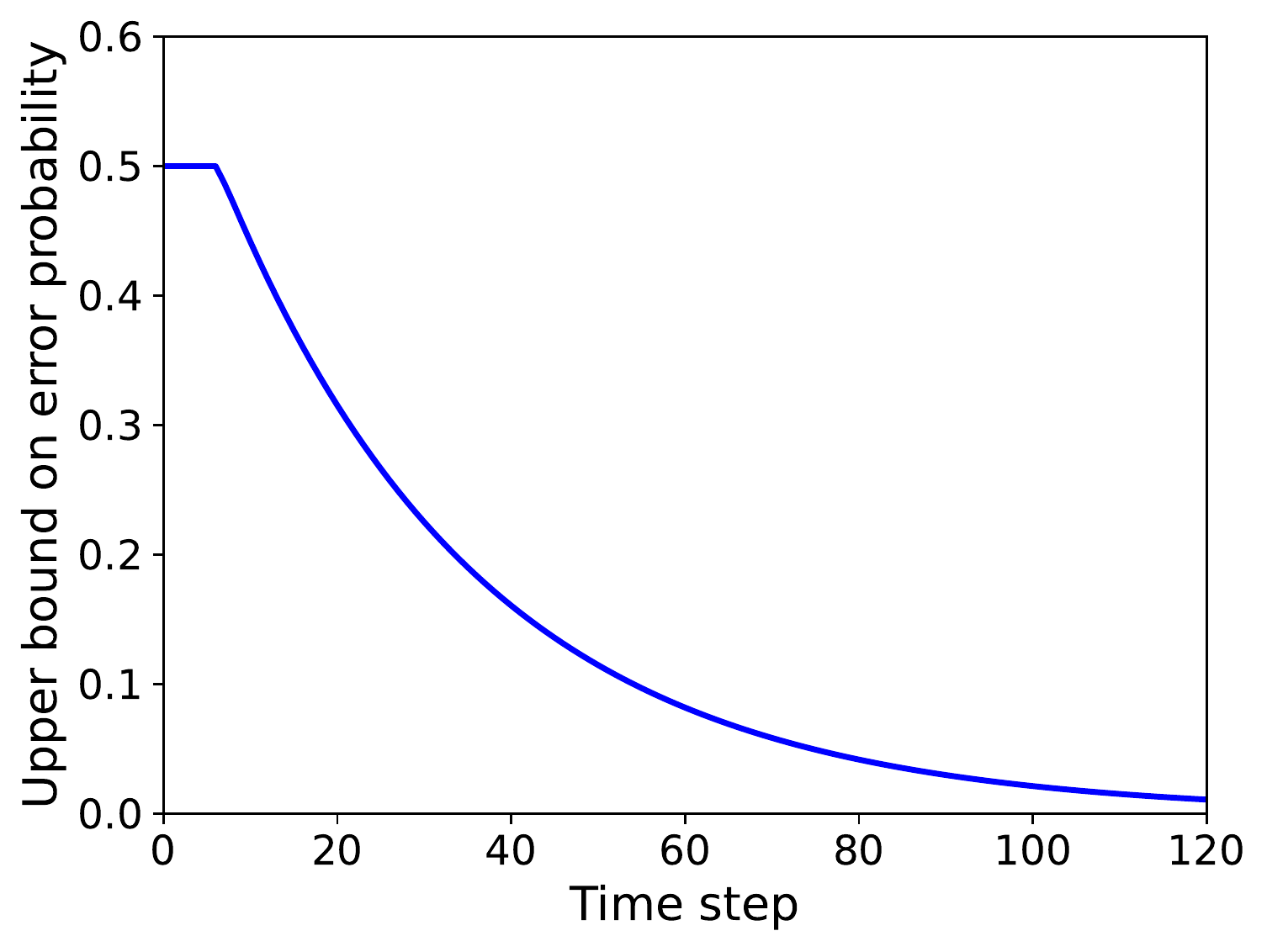}
	\caption{The upper bound for the error probability of detecting the agent type in the intrusion detection.}\label{fig:upperbound}
\end{figure}

\subsection{MDP-based recommendation systems}
In an MDP-based recommendation system \cite{GS-DH-RIB:05, KC-MC-DK-PN-AR:20}, the items recommended to customers are strategically selected to account for recommendations' long-term effects. However, a single MDP model may not be adequate to capture different types of customers' purchasing behaviors. This section shows how the algorithm developed in Section~\ref{sec:general} can be applied to design a recommendation strategy that identifies the customer type based on the observed purchasing sequence.

We consider a recommendation system with $10$ items, and the system selects one item to recommend to a customer at each step. The state space of the MDP consists of all possible ordered past purchase histories of length two, resulting in 111 total states (including one state representing the empty purchase history and $10$ states representing histories of single purchases). We further consider $N=6$ customer types, each of which has a randomly generated preference ranking over the items (the first item on the preference list is the most preferred). Let $v\in\Delta_{10}$ be a probability vector uniformly sampled from the probability simplex $\Delta_{10}$. When there are no recommendations, a customer buys the $i$-th ranked item on his/her preference list with probability being equal to the $i$-th largest element in $v$.  When an item is recommended, the probability of buying the recommended item increases by a multiplicative factor of $1+\alpha_k$ where $k\in\mathbb{N}_{1}^N$ is the customer type and $0\leq\alpha_k\leq\frac{1}{\max_{i}\{v_i\}}-1$ models the customer's sensitivity to recommendations; the probabilities of buying other non-recommended items are scaled down accordingly. We also introduce one identity-revealing transition for each customer type by setting the probability of buying the lowest ranked item under some recommendation at some state to be zero, where the recommendation and the state are randomly selected. To make the detection process slightly more difficult, we assume that the customers have the same sensitivity parameter $\alpha_k=\alpha$ for $k\in\mathbb{N}_{1}^N$, and we generate $\alpha$ uniformly randomly from $[0,\frac{1}{\max_{i}\{v_i\}}-1]$.

\begin{figure*}[http]
	\centering
	\includegraphics[scale=0.5]{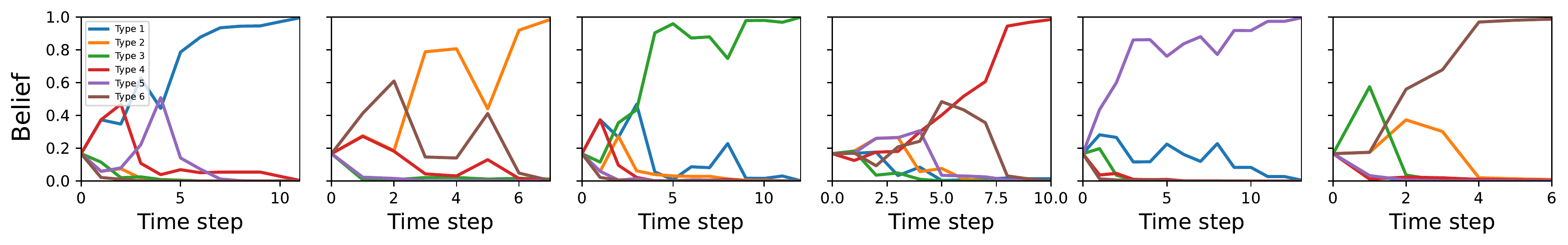}
	\caption{The evolution of beliefs over customer types in the MDP-based recommendation system.}\label{fig:evolutionMDP}
\end{figure*}

Fig.~\ref{fig:evolutionMDP} shows the evolution of beliefs over the customer types where each subplot corresponds to one specific customer type being the ground truth MDP. In all the realized scenarios, the recommendation system successfully detects the customer type after observing the customers' reactions to a few recommendations. Similar to the binary case, the upper bound on the error probability of the Bayesian detection goes to zero exponentially fast as shown in~Fig.~\ref{fig:upperbound_rec} (we cap the upper bound at $1$ when the upper bound is greater than $1$).

\begin{figure}[http]
	\centering
	\includegraphics[scale=0.44]{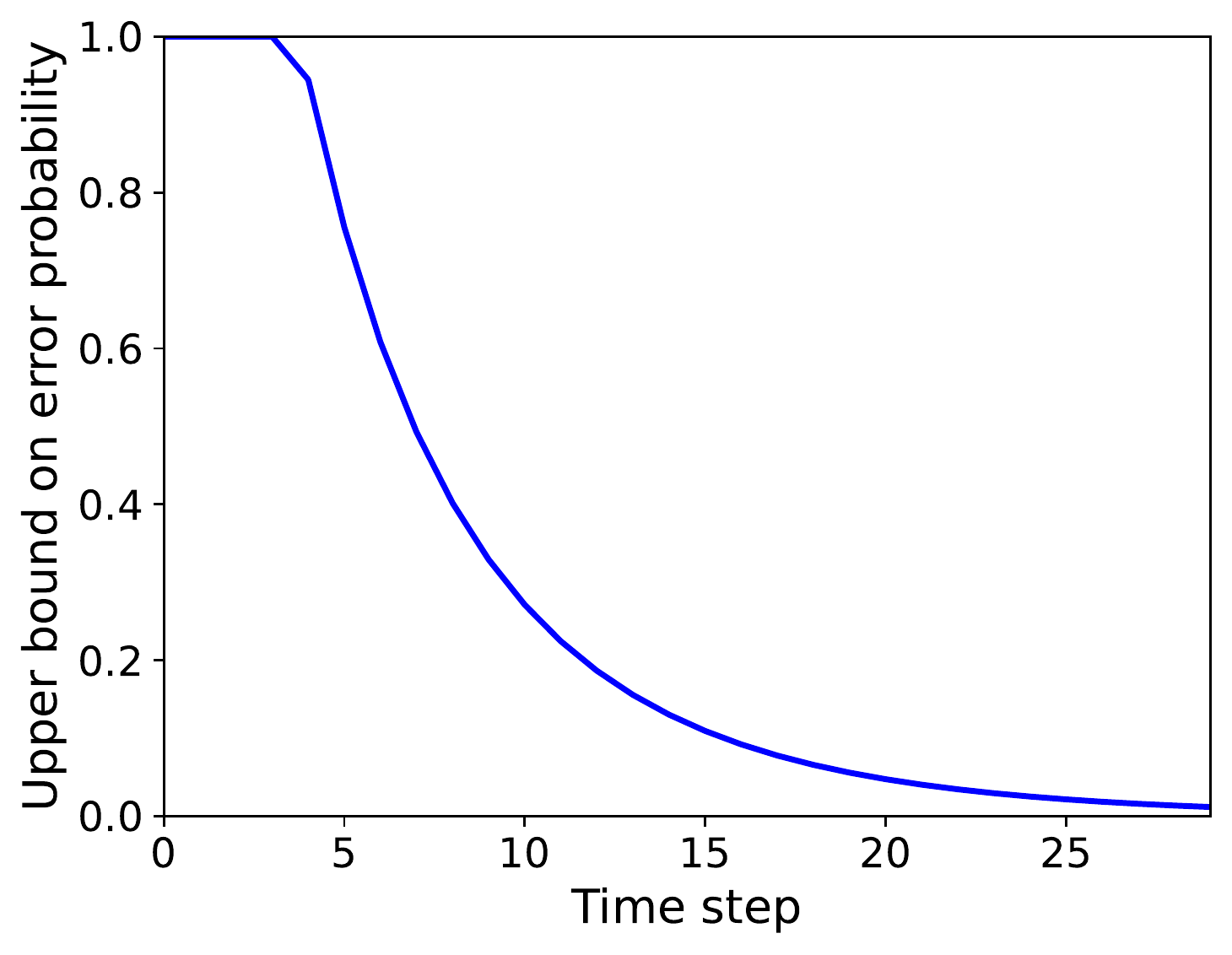}
	\caption{The upper bound for the error probability of detecting the customer type in the MDP-based recommendation system.}\label{fig:upperbound_rec}
\end{figure}

\section{Conclusion}\label{sec:conclusion}
We studied the policy synthesis problem for achieving asymptotically perfect detection (APD) for multi-model MDPs (MMDPs).  We started with the binary case where the MMDPs consist of two MDPs and derived a necessary and sufficient condition for the existence of policies that achieve APD. We then developed an efficient polynomial-time algorithm that synthesizes policies that achieve APD or determines they do not exist. We finally extended the results to the general case of MMDPs and proposed a similar policy synthesis algorithm.

For future work, we will investigate the intrinsic complexity of the APD problem for general MMDPs. On the other hand, APD might be a too strong requirement for policies to exist, and we will explore other appropriate notions of detection.

\appendix
\subsection{Proof of Theorem~\ref{thm:main}}\label{sec:proofmainthm}

Since the detection problem for $\mathcal{M}$ is the equivalent to that for $\mathcal{M}^{\p}$, we will focus on $\mathcal{M}^{\p}$ in the following. With a slight abuse of notation, we will use $\mathcal{P}_{i}^{\bm{\pi}}$ to also denote the probability measure induced by $\delta_i^{\p}$ and the policy $\bm{\pi}$ over the measurable space  $(\mathcal{H}^{\I},\mathcal{Q}^{\I})$ for $i\in\{1,2\}$.

Before diving into the proof of Theorem~\ref{thm:main}, we first introduce some notation and useful lemmas. We partition the set of informative state-action pairs into disjoint subsets as follows
\begin{equation*}
	\ISA^{\p}=\ISA^{\p}_{\textup{a}}\cup\ISA^{\p}_{0}\cup\ISA^{\p}_{1}\cup\ISA^{\p}_{2}\cup\ISA^{\p}_{3},
\end{equation*}
where 
\begin{align*}
	\ISA^{\p}_{\textup{a}}&=\{(\bot_1,a^{\bot_1}),(\bot_2,a^{\bot_2})\},\\		\ISA^{\p}_{0}&=\{(s,a)\in\ISA^{\p}\,|\,\delta^{\p}_{1}(\bot_1\,|\,s,a)=0,\delta^{\p}_{2}(\bot_2\,|\,s,a)=0\},\\
	\ISA^{\p}_{1}&=\{(s,a)\in\ISA^{\p}\,|\,\delta^{\p}_{1}(\bot_1\,|\,s,a)>0,\delta^{\p}_{2}(\bot_2\,|\,s,a)=0\},\\
	\ISA^{\p}_{2}&=\{(s,a)\in\ISA^{\p}\,|\,\delta^{\p}_{1}(\bot_1\,|\,s,a)=0,\delta^{\p}_{2}(\bot_2\,|\,s,a)>0\},\\
	\ISA^{\p}_{3}&=\{(s,a)\in\ISA^{\p}\,|\,\delta^{\p}_{1}(\bot_1\,|\,s,a)>0,\delta^{\p}_{2}(\bot_2\,|\,s,a)>0\}.
\end{align*}
Moreover, we let $\ISA^{\p}_{\textup{tran}} = \ISA^{\p}_{1}\cup\ISA^{\p}_{2}\cup\ISA^{\p}_{3}$, $\ISA^{\p}_{\textup{tran},1} = \ISA^{\p}_{1}\cup\ISA^{\p}_{3}$ and $\ISA^{\p}_{\textup{tran},2} = \ISA^{\p}_{2}\cup\ISA^{\p}_{3}$.

We next introduce the concept of orthogonal measures. It turns out that the BC between two probability measures is zero if and only if these two measures are orthogonal. 

\begin{defi}[Orthogonal measures {\cite[Section 2]{SK:48}}]
Let $(\Omega,\mathcal{F})$ be a measurable space. Two probability measures $\mu_1$ and $\mu_2$ over $(\Omega,\mathcal{F})$ are orthogonal if there exists a measurable set $F\in\mathcal{F}$ such that $\mu_1(F)=1$ and $\mu_2(\overline{F})=1$.
\end{defi}


The following lemma connects the orthogonality of measures with the BC.
\begin{lema}[BC and orthogonal measures {\cite[Section 4]{SK:48}}]\label{lema:bcorthogonal}
Let $(\Omega,\mathcal{F})$ be a measurable space.  The BC between two probability measures $\mu_1$ and $\mu_2$ over $(\Omega,\mathcal{F})$ is zero if and only if they are orthogonal.
\end{lema}

If we can show that the probability measures $\mathcal{P}_{1}^{\bm{\pi}}$ and $\mathcal{P}_{2}^{\bm{\pi}}$ are orthogonal under the policy $\bm{\pi}$, then by Lemma~\ref{lema:bcorthogonal}, we know that the BC between $\mathcal{P}_{1}^{\bm{\pi}}$ and $\mathcal{P}_{2}^{\bm{\pi}}$ is zero. Further by Lemma~\ref{lema:APDgeneral}, we conclude that APD is achieved under the policy $\bm{\pi}$. The next two lemmas reveal some relationships among the probability measures $\mathcal{P}_{1}^{\bm{\pi}}$, $\mathcal{P}_{2}^{\bm{\pi}}$ and $\mathcal{P}_{\I}^{\bm{\pi}}$ for any given policy $\bm{\pi}$.

\begin{lema}[Infinitely often visited informative state-action pairs]\label{lemma:probmeasures}
Given a binary MMDP $\mathcal{M}=\{M_1,M_2\}$, any of its informative MDP $M^{\I}$ and a policy $\bm{\pi}$, let $\mathcal{H}_{\D}=\{h\in \mathcal{H}^{\I}\,|\, \inft(h)\cap{\ISA^{\p}}\neq\emptyset\}$, then the following statements hold
\begin{enumerate}[label=(\roman*)]
\item\label{itm:probmeasure} $\mathcal{P}_{\I}^{\bm{\pi}}(\mathcal{H}_{\D})=1$ if and only if $\mathcal{P}_{1}^{\bm{\pi}}(\mathcal{H}_{\D})=1$;
\item $\mathcal{P}_{\I}^{\bm{\pi}}(\mathcal{H}_{\D})=1$ if and only if $\mathcal{P}_{2}^{\bm{\pi}}(\mathcal{H}_{\D})=1$.
\end{enumerate}
\end{lema}
\begin{proof}
By symmetry, we only need to prove~\ref{itm:probmeasure}. Instead of proving~\ref{itm:probmeasure} directly, we prove the equivalent statement that $\mathcal{P}_{\I}^{\bm{\pi}}(\overline{\mathcal{H}_{\D}})>0$ if and only if $\mathcal{P}_{1}^{\bm{\pi}}(\overline{\mathcal{H}_{\D}})>0$.

{$\mathcal{P}_{\I}^{\bm{\pi}}(\overline{\mathcal{H}_{\D}})>0\implies \mathcal{P}_{1}^{\bm{\pi}}(\overline{\mathcal{H}_{\D}})>0$}: We first write $\overline{\mathcal{H}_{\D}}$ out more explicitly. For $t\in\mathbb{N}_{\geq0}$, let $E_t=\{h\in \mathcal{H}^{\I}\,|\,(h(t),h[t])\in\ISA^{\p}\}$, then we have $\overline{\mathcal{H}_{\D}}=\cup_{t=0}^\infty\cap_{\tau=t}^\infty \overline{E_\tau}$. Since $\cap_{\tau=t}^\infty \overline{E_\tau}$ is an increasing sequence of sets indexed by $t$, i.e., $\cap_{\tau=t}^\infty \overline{E_\tau}\subset\cap_{\tau=t+1}^\infty \overline{E_\tau}$, by the continuity of probability measures, we have that $\mathcal{P}_{\I}^{\bm{\pi}}(\overline{\mathcal{H}_{\D}})=\lim_{t\rightarrow\infty}\mathcal{P}_{\I}^{\bm{\pi}}(\cap_{\tau=t}^\infty \overline{E_\tau})>0$. Therefore, there exists $\tilde{t}\in\mathbb{N}_{\geq0}$, such that $\mathcal{P}_{\I}^{\bm{\pi}}(\cap_{\tau=\tilde{t}}^\infty \overline{E_\tau})>0$. 

In the following, we show that $\mathcal{P}_{1}^{\bm{\pi}}(\cap_{\tau=\tilde{t}}^\infty \overline{E_\tau})>0$, which implies that  $\mathcal{P}_{1}^{\bm{\pi}}(\overline{\mathcal{H}_{\D}})=\lim_{t\rightarrow\infty}\mathcal{P}_{1}^{\bm{\pi}}(\cap_{\tau=t}^\infty \overline{E_\tau})>0$.

When $\tilde{t}=0$, we have that $\mathcal{P}_{\I}^{\bm{\pi}}(\cap_{\tau=0}^\infty \overline{E_\tau})>0$, i.e., the set of histories that do not contain the informative state action pairs has positive measure under $\mathcal{P}_{\I}^{\bm{\pi}}$. Since the policy as well as the transition probabilities are the same for those histories in $M_1^{\p}$ and $M^{\I}$, we conclude that $\mathcal{P}_{1}^{\bm{\pi}}(\cap_{\tau=0}^\infty \overline{E_\tau})>0$.

Suppose $\tilde{t}\geq1$. 
Let 
\begin{equation}\label{eq:setofprefices}
	{\mathcal{L}}_{\tilde{t}}=\{(s_0,\cdots, a_{\tilde{t}-1})\,|\,s_\tau\in\mathcal{S}^{\p},a_{\tau}\in\mathcal{A}_{s_\tau}^{\p}\textup{ for }\tau\in\mathbb{N}_{0}^{\tilde{t}-1}\}.
\end{equation}
Note that ${\mathcal{L}}_{\tilde{t}}$ is a finite set. Then, by the total probability formula, we have
\begin{equation*}
	\mathcal{P}_{\I}^{\bm{\pi}}(\cap_{\tau=\tilde{t}}^\infty \overline{E_\tau})=
	\sum_{{\ell}_{\tilde{t}}\in{\mathcal{L}}_{\tilde{t}}}\mathcal{P}_{\I}^{\bm{\pi}}(\cap_{\tau=\tilde{t}}^\infty \overline{E_\tau}\,|\,\ell_{\tilde{t}})\mathcal{P}_{\I}^{\bm{\pi}}(\ell_{\tilde{t}})>0.
\end{equation*}
Thus, there must exist some $\tilde{\ell}_{\tilde{t}}\in{\mathcal{L}}_{\tilde{t}}$ such that
\begin{equation}\label{eq:singlepositive}
	\mathcal{P}_{\I}^{\bm{\pi}}(\tilde{\ell}_{\tilde{t}})\mathcal{P}_{\I}^{\bm{\pi}}(\cap_{\tau=\tilde{t}}^\infty \overline{E_\tau}\,|\,\tilde{\ell}_{\tilde{t}})>0.
\end{equation}
For~\eqref{eq:singlepositive} to hold true, $\tilde{\ell}_{\tilde{t}}$ must not contain $\bot_{i}$ for $i\in\{1,2\}$ since otherwise $\mathcal{P}_{\I}^{\bm{\pi}}(\cap_{\tau=\tilde{t}}^\infty \overline{E_\tau}\,|\,\tilde{\ell}_{\tilde{t}})=0$. Therefore, from the construction of the informative MDP, we also have $\mathcal{P}_{1}^{\bm{\pi}}(\tilde{\ell}_{\tilde{t}})>0$. On the other hand, note that all histories in $\cap_{\tau=\tilde{t}}^\infty \overline{E_\tau}$ with the first $2(\tilde{t}-1)$ elements coinciding with $\tilde{\ell}_{\tilde{t}}$ do not contain any informative state-action pairs after step $\tilde{t}$, and the policy as well as the transition probabilities are the same for those histories in $M_1^{\p}$ and $M^{\I}$. Thus, we have $\mathcal{P}_{1}^{\bm{\pi}}(\cap_{\tau=\tilde{t}}^\infty \overline{E_\tau}\,|\,\tilde{\ell}_{\tilde{t}})=\mathcal{P}_{\I}^{\bm{\pi}}(\cap_{\tau=\tilde{t}}^\infty \overline{E_\tau}\,|\,\tilde{\ell}_{\tilde{t}})>0$. In summary, we have $\mathcal{P}_{1}^{\bm{\pi}}(\cap_{\tau=\tilde{t}}^\infty \overline{E_\tau})>0$, which implies $\mathcal{P}_{1}^{\bm{\pi}}(\overline{\mathcal{H}_{\D}})>0$.

The converse {$\mathcal{P}_{1}^{\bm{\pi}}(\overline{\mathcal{H}_{\D}})>0\implies \mathcal{P}_{\I}^{\bm{\pi}}(\overline{\mathcal{H}_{\D}})>0$} can be shown in an almost identical manner and is omitted here in the interest of brevity.


%
%
\end{proof}

\begin{lema}[Finitely often visited informative state-action pairs]\label{lemma:fnintelyoften}
	Given a binary MMDP $\mathcal{M}=\{M_1,M_2\}$ and a policy $\bm{\pi}$, for $i\in\{1,2\}$, we have $\mathcal{P}_{i}^{\bm{\pi}}(\{h\in \mathcal{H}^{\I}\,|\, \inft(h)\cap{\ISA^{\p}_{\textup{tran},i}}\neq\emptyset\})=0$.
\end{lema}
\begin{proof}	
We first show that any state-action pair $(s,a)\in{\ISA^{\p}_{\textup{tran},i}}$ does not belong to any end component of $M_i^{\p}$ by contradiction. Let $C\in\mathcal{C}(M_i^{\p})$ be an end component and suppose $(s,a)\in{\ISA^{\p}_{\textup{tran},i}}$ and $(s,a)\in C$. Since $\delta_i^{\p}(\bot_i\,|\,s,a)>0$ by the definition of ${\ISA^{\p}_{\textup{tran},i}}$, we must have $\bot_i\in C$ by Definition~\ref{def:MEC}\ref{itm:nextstates}. However, this violates Definition~\ref{def:MEC}\ref{itm:strongconnect} for $C$ since $s$ is not reachable from $\bot_i$ in $C$. Therefore, $C$ cannot be an end component, which is a contradiction.

By~\cite[Theorem 3.2]{LdeA:97}, we know that the set of infinitely often visited state-action pairs constitute an end component almost surely, i.e., 
\begin{equation}\label{eq:iovisitEC}
\mathcal{P}_{i}^{\bm{\pi}}(\{h\in\mathcal{H}^{\I}\,|\,\inft(h)\textup{ is an end component}\})=1.
\end{equation}
Since any state-action pair $(s,a)\in{\ISA^{\p}_{\textup{tran},i}}$ cannot be in any end component of $M_i^{\p}$, we also have 
\begin{multline}\label{eq:ECandtran}
\{h\in\mathcal{H}^{\I}\,|\,\inft(h)\textup{ is an end component}\}\cap\\\{h\in \mathcal{H}^{\I}\,|\, \inft(h)\cap{\ISA^{\p}_{\textup{tran},i}}\neq\emptyset\}=\emptyset.
\end{multline}
Combining~\eqref{eq:iovisitEC} and~\eqref{eq:ECandtran}, we conclude that $\mathcal{P}_{i}^{\bm{\pi}}(\{h\in \mathcal{H}^{\I}\,|\, \inft(h)\cap{\ISA^{\p}_{\textup{tran},i}}\neq\emptyset\})=0$.

\end{proof}

Now we are ready to prove Theorem~\ref{thm:main}.
\begin{proof}[Proof of Theorem~\ref{thm:main}]
Let $\mathcal{H}_{\D}=\{h\in \mathcal{H}^{\I}\,|\, \inft(h)\cap{\ISA^{\p}}\neq\emptyset\}\subset\mathcal{H}^{\I}$.

\subsubsection{If \texorpdfstring{$\mathcal{P}_{\I}^{\bm{\pi}}(\mathcal{H}_{\D})=1$}{TEXT}, then APD is achieved under the policy \texorpdfstring{$\bm{\pi}$}{TEXT}} Since $\mathcal{P}_{\I}^{\bm{\pi}}(\mathcal{H}_{\D})=1$, by Lemma~\ref{lemma:probmeasures}, we also have $\mathcal{P}_{i}^{\bm{\pi}}(\mathcal{H}_{\D})=1$ for $i\in\{1,2\}$. In the following, we show that the probability measures $\mathcal{P}_{1}^{\bm{\pi}}$ and $\mathcal{P}_{2}^{\bm{\pi}}$ are orthogonal by constructing a measurable set $F\in\mathcal{Q}$ such that $\mathcal{P}_{1}^{\bm{\pi}}(F)=1$ and $\mathcal{P}_{2}^{\bm{\pi}}(\overline{F})=1$. Then, the conclusion follows from Lemma~\ref{lema:bcorthogonal} and Lemma~\ref{lema:APDgeneral}. 

Let $\mathcal{H}_{\D}=F_1\cup F_2\cup G$ where for $i\in\{1,2\}$,
\begin{equation*}
	F_i=\cup_{t=0}^\infty\cap_{\tau=t}^\infty\{h\in \mathcal{H}_{\D}\,|\, (h(\tau),h[\tau])=(\bot_i,a^{\bot_i})\},
\end{equation*}
and 
\begin{equation*}
G=\mathcal{H}_{\D}\setminus (F_1\cup F_2)=\{h\in \mathcal{H}_{\D}\,|\, \inft(h)\cap({\ISA^{\p}}\setminus	\ISA^{\p}_{\textup{a}})\neq\emptyset\}. 
\end{equation*}
Note that the sets $F_1$, $F_2$ and $G$ are disjoint and measurable. Moreover, since for $i\in\{1,2\}$, $\mathcal{P}_{i}^{\bm{\pi}}(\mathcal{H}_{\D})=1$, and $\bot_i$ is not reachable in $M_{3-i}^{\p}$, i.e., $\mathcal{P}_{i}^{\bm{\pi}}(F_{3-i})=0$, we have that $\mathcal{P}_{i}^{\bm{\pi}}(F_i\cup G)=1$. 
If $G=\emptyset$, then $\mathcal{H}_{\D}=F_1\cup F_2$. Let $F=F_1$ and $F_2\subset\overline{F}=\mathcal{H}\setminus F_1$, we have $\mathcal{P}_{1}^{\bm{\pi}}(F)=\mathcal{P}_{1}^{\bm{\pi}}(F_1)=1$ and $1\geq\mathcal{P}_{2}^{\bm{\pi}}(\overline{F})\geq\mathcal{P}_{2}^{\bm{\pi}}(F_2)=1$. 

Suppose $G\neq\emptyset$. We further partition $G$ as
\begin{equation*}
G =G_0\cup G_1\cup G_2,
\end{equation*}
where 
\begin{align*}
G_{0}&=\{h\in G\,|\, \inft(h)\cap\ISA^{\p}_{0}\neq\emptyset, \inft(h)\cap\ISA^{\p}_{\textup{tran}}=\emptyset\},\\
G_1&=\{h\in G\,|\,\inft(h)\cap\ISA^{\p}_{2}\neq\emptyset, \inft(h)\cap\ISA^{\p}_{\textup{tran},1}=\emptyset\},\\
G_2&=G\setminus(G_0\cup G_1).
\end{align*}
Note that $G_0$, $G_1$ and $G_2$ are disjoint, and $G_2\subset\{h\in G\,|\,\inft(h)\cap\ISA^{\p}_{\textup{tran},1}\neq\emptyset\}$. 
If $G_0=\emptyset$, then $G=G_1\cup G_2$ and $\mathcal{H}_{\D}=F_1\cup F_2\cup G_1 \cup G_2$. Since $\mathcal{P}_{1}^{\bm{\pi}}(G_2)=0$ and $\mathcal{P}_{2}^{\bm{\pi}}(G_1)=0$ by Lemma~\ref{lemma:fnintelyoften}, for $F=F_1\cup G_1$ and $(F_2\cup G_2)\subset\overline{F}=\mathcal{H}\setminus F$, we have that $\mathcal{P}_{1}^{\bm{\pi}}(F)=\mathcal{P}_{1}^{\bm{\pi}}(F_1\cup G_1)=1$ and $1\geq\mathcal{P}_{2}^{\bm{\pi}}(\overline{F})\geq\mathcal{P}_{2}^{\bm{\pi}}(F_2\cup G_2)=1$.

Suppose $G_0\neq\emptyset$. We further write $G_0$ as
\begin{equation}\label{eq:decomposeF0}
	G_{0}=\cup_{\SA\in2^{\ISA^{\p}_{0}}}G_0^{\SA},
\end{equation}
where $2^{\ISA^{\p}_{0}}$ is the power set of ${\ISA^{\p}_{0}}$ and
\begin{equation*}
G_0^{\SA}=\{h\in G_0\,|\,\inft(h)\cap\ISA^{\p}_{0}=\SA\}.
\end{equation*}
Clearly, the sets $G_0^{\SA}$'s on the right-hand side of~\eqref{eq:decomposeF0} are disjoint. By the strong law of large numbers, we have that $\mathcal{P}_{1}^{\bm{\pi}}(G_0^{\SA})=\mathcal{P}_{1}^{\bm{\pi}}(\tilde{G_0^{\SA}})$ and $\mathcal{P}_{1}^{\bm{\pi}}(G_0^{\SA}\setminus\tilde{G_0^{\SA}})=0$,
where
\begin{multline*}
	\tilde{G_0^{\SA}}=\\\{h\in G_0\,|\,\inft(h)\cap\ISA^{\p}_{0}=\SA,\forall(s,a)\in\SA, \forall s'\in\mathcal{S}^{\p},\\
	\lim_{t\rightarrow\infty}\frac{\sum_{\tau=0}^t\mathbf{1}_{\{(h(\tau),h[\tau],h(\tau+1))=(s,a,s')\}}(h)}{t+1}=\delta_1^{\p}({s'\,|\,s,a})\}.
\end{multline*}
At the same time, since $\delta_1^{\p}(\cdot\,|\,s,a)\neq\delta_2^{\p}(\cdot\,|\,s,a)$ for $(s,a)\in\ISA^{\p}_{0}$, we also have that $	\mathcal{P}_{2}^{\bm{\pi}}(\tilde{G_0^{\SA}})=0$.

Therefore, we have that
\begin{equation*}
\mathcal{H}_{\D}=F_1\cup F_2\cup G_1 \cup G_2\cup(\cup_{\SA\in2^{\ISA^{\p}_{0}}}{G_0^{\SA}})
\end{equation*}
Let $F=F_1\cup G_1\cup(\cup_{\SA\in2^{\ISA^{\p}_{0}}}\tilde{G_0^{\SA}})$, we have that $\mathcal{P}_{1}^{\bm{\pi}}(F)=\mathcal{P}_{1}^{\bm{\pi}}(F_1\cup G_1\cup(\cup_{\SA\in2^{\ISA^{\p}_{0}}}\tilde{G_0^{\SA}}))=1$ and $1\geq\mathcal{P}_{2}^{\bm{\pi}}(\overline{F})\geq\mathcal{P}_{2}^{\bm{\pi}}(F_2\cup G_2\cup(\cup_{\SA\in2^{\ISA^{\p}_{0}}}(G_0^{\SA}\setminus\tilde{G_0^{\SA}})))=1$.

\subsubsection{If \texorpdfstring{$\mathcal{P}_{\I}^{\bm{\pi}}(\mathcal{H}_{\D})<1$}{TEXT}, then APD is not achieved under the policy \texorpdfstring{$\bm{\pi}$}{TEXT}} In this case, we have $\mathcal{P}_{\I}^{\bm{\pi}}(\overline{\mathcal{H}_{\D}})>0$. Let $E_t=\{h\in \mathcal{H}^{\I}\,|\,(h(t),h[t])\in\ISA^{\p}\}$, then by Lemma~\ref{lemma:probmeasures} and the proof therein, we know that there exists $\tilde{t}\in\mathbb{N}_{\geq0}$ such that $\mathcal{P}_{1}^{\bm{\pi}}(\cap_{\tau=\tilde{t}}^\infty \overline{E_\tau})>0$ and $\mathcal{P}_{2}^{\bm{\pi}}(\cap_{\tau=\tilde{t}}^\infty \overline{E_\tau})>0$. Moreover, there exists an $\tilde{\ell}_{\tilde{t}}\in\mathcal{L}_{\tilde{t}}$, where $\mathcal{L}_{\tilde{t}}$ is defined in~\eqref{eq:setofprefices}, such that
\begin{equation*}
	\mathcal{P}_{1}^{\bm{\pi}}(\tilde{\ell}_{\tilde{t}})\mathcal{P}_{1}^{\bm{\pi}}(\cap_{\tau=\tilde{t}}^\infty \overline{E_\tau}\,|\,\tilde{\ell}_{\tilde{t}})>0.
\end{equation*}
Since $\tilde{\ell}_{\tilde{t}}$ must not contain $\bot_1$ (otherwise we have $\mathcal{P}_{1}^{\bm{\pi}}(\cap_{\tau=\tilde{t}}^\infty \overline{E_\tau}\,|\,\tilde{\ell}_{\tilde{t}})=0$), we also have $	\mathcal{P}_{2}^{\bm{\pi}}(\tilde{\ell}_{\tilde{t}})>0$. Moreover, since all histories in $\cap_{\tau=\tilde{t}}^\infty \overline{E_\tau}$ with the first $2(\tilde{t}-1)$ elements being $\tilde{\ell}_{\tilde{t}}$ do not contain any informative state-action pairs after step $\tilde{t}$, and the policy as well as the transition probabilities are the same for those histories in $M_1^{\p}$ and $M_{2}^{\p}$, we also have $\mathcal{P}_{2}^{\bm{\pi}}(\cap_{\tau=\tilde{t}}^\infty \overline{E_\tau}\,|\,\tilde{\ell}_{\tilde{t}})>0$. Note that $\cap_{\tau=\tilde{t}}^t\overline{E_{\tau}}$ is a decreasing sequence of events as $t$ increases, thus we must have $\mathcal{P}_{i}^{\bm{\pi}}(\cap_{\tau=\tilde{t}}^t \overline{E_\tau}\,|\,\tilde{\ell}_{\tilde{t}})>0$ for all $t\geq\tilde{t}$ and $i\in\{1,2\}$. For $t\geq\tilde{t}$, let
 \begin{align*}
	{\mathcal{L}}_{t}=\{(s_0,\cdots,s_t, a_{t})\,|\,&(s_0,\cdots,a_{\tilde{t}-1})=\tilde{\ell}_{\tilde{t}}, \textup{ and } \\
	& (s_{\tau},a_{\tau})\notin\ISA^{\p}\textup{ for }\tau\in\mathbb{N}_{\tilde{t}}^{t}\}.
\end{align*}

We now note that the BC satisfies for all $t\geq\tilde{t}$, 
\begin{align}\label{eq:BCwithtime}
	\begin{split}
&\quad B(t+1,\bm{\pi})\\
&=\sum_{h_{{t+1}}\in\mathcal{H}^{\I}_{{t+1}}}\sqrt{\mathbb{P}_{1}^{\bm{\pi}}(h_{{t+1}})\mathbb{P}_{2}^{\bm{\pi}}(h_{{t+1}})}\\
&\geq\sum_{\ell_{{t}}\in\mathcal{L}_{{t}}}\sqrt{\mathbb{P}_{1}^{\bm{\pi}}(\ell_{{t}})\mathbb{P}_{2}^{\bm{\pi}}(\ell_{{t}})\sum_{s\in\mathcal{S}^{\p}}\delta_1^{\p}(s\,|\,s_{t},a_{t})\delta_2^{\p}(s\,|\,s_{t},a_{t})}\\
&=\sum_{\ell_{{t}}\in\mathcal{L}_{{t}}}\sqrt{\mathbb{P}_{1}^{\bm{\pi}}(\ell_{{t}})\mathbb{P}_{2}^{\bm{\pi}}(\ell_{{t}})}\\
&= \sqrt{\mathcal{P}_{1}^{\bm{\pi}}(\tilde{\ell}_{\tilde{t}})\mathcal{P}_{1}^{\bm{\pi}}(\cap_{\tau=\tilde{t}}^t \overline{E_\tau}\,|\,\tilde{\ell}_{\tilde{t}})\mathcal{P}_{2}^{\bm{\pi}}(\tilde{\ell}_{\tilde{t}})\mathcal{P}_{2}^{\bm{\pi}}(\cap_{\tau=\tilde{t}}^t \overline{E_\tau}\,|\,\tilde{\ell}_{\tilde{t}})}>0,
\end{split}
\end{align}
where the second and third equalities follow from the fact that for any $\ell_t\in\mathcal{L}_t$ and $\tau\geq\tilde{t}$, we have that $(s_{\tau},a_{\tau})\notin \ISA^{\p}$. Take the limit as $t$ goes to infinity on both sides of~\eqref{eq:BCwithtime}, we have that
\begin{multline*}
B(\bm{\pi})\geq\\\sqrt{	\mathcal{P}_{1}^{\bm{\pi}}(\tilde{\ell}_{\tilde{t}})\mathcal{P}_{1}^{\bm{\pi}}(\cap_{\tau=\tilde{t}}^\infty \overline{E_\tau}\,|\,\tilde{\ell}_{\tilde{t}})	\mathcal{P}_{2}^{\bm{\pi}}(\tilde{\ell}_{\tilde{t}})\mathcal{P}_{2}^{\bm{\pi}}(\cap_{\tau=\tilde{t}}^\infty \overline{E_\tau}\,|\,\tilde{\ell}_{\tilde{t}})}>0.
\end{multline*}
Thus, by Lemma~\ref{lema:APDgeneral}, APD is not achieved under the policy $\bm{\pi}$.

\end{proof}

\subsection{Proof of Theorem~\ref{thm:alg}}\label{sec:proofalg}

We will use the following lemma in our proof.

\begin{lema}[MECs, reachability probability, and infinitely often visited states]\label{lemma:reachabilityandio}
Given an MDP $M=(\mathcal{S},\mathcal{A},\delta,{s}_{\init})$, let $\mathcal{C}(M)$ be the set of MECs, $\SA^{\textup{target}}=\{(s,a)\,|\,s\in\mathcal{S},a\in\mathcal{A}_s\}$ be the set of target state-action pairs, and $\mathcal{C}^{\textup{target}}(M)=\{(\mathcal{X},\mathcal{U})\in\mathcal{C}(M)\,|\,\exists (s,a)\in\SA^{\textup{target}},s\in \mathcal{X},a\in \mathcal{U}\}$ be the set of target MECs. Then, 
\begin{multline*}
\mathbb{P}^{\max}_{s_{\init},M}(\textup{reach}(\mathcal{C}^{\textup{target}}(M)))=\\\max_{\bm{\pi}}\mathcal{P}^{\bm{\pi}}(\{h\in\mathcal{H}\,|\,\inft(h)\cap\SA^{\textup{target}}\neq\emptyset\}).
\end{multline*}
\end{lema}
\begin{proof}
The result follows directly from \cite[Theorem 4.2]{LdeA:97}.
\end{proof}

\begin{proof}[Proof of Theorem~\ref{thm:alg}]
Let $\mathcal{H}_{\D}=\{h\in \mathcal{H}^{\I}\,|\, \inft(h)\cap{\ISA^{\p}}\neq\emptyset\}\subset\mathcal{H}^{\I}$.

We first show that if Algorithm~\ref{alg:APD} reports no solution, then APD cannot be achieved.
In this case,  $\mathbb{P}_{s_{\init},M^{\I}}^{\textup{max}}(\textup{reach}(\mathcal{C}^{\I}(M^{\I})))<1$. Then, by Lemma~\ref{lemma:reachabilityandio}, we have that $\mathcal{P}_{\I}^{\bm{\pi}}(\{h\in \mathcal{H}^{\I}\,|\, \inft(h)\cap{\ISA^{\p}}\neq\emptyset\})<1$ for any policy $\bm{\pi}$. Therefore, by Theorem~\ref{thm:main}, no policy that achieves APD exists.

We next show that the policy synthesized by Algorithm~\ref{alg:APD} indeed achieves APD.
Note that the policy $\bm{\pi}^0$ achieves the reachability probability  $\mathbb{P}_{s_{\init},M^{\I}}^{\bm{\pi}^0}(\textup{reach}(\mathcal{C}^{\I}(M^{\I})))=\mathbb{P}_{s_{\init},M^{\I}}^{\textup{max}}(\textup{reach}(\mathcal{C}^{\I}(M^{\I})))=1$ outside of the informative MECs, and under $\bm{\pi}^{C}$ for any informative  MEC $C\in\mathcal{C}^{\I}(M^{\I})$, we have that $\mathcal{P}^{\bm{\pi}^{C}}_{\I}(\mathcal{H}_{\D}\,|\,\textup{reach}(C))=1$. Therefore,  we conclude that
 \begin{multline*}
     \mathcal{P}^{\bm{\pi}}_{\I}(\mathcal{H}_{\D})\\=\sum_{C\in\mathcal{C}^{\I}(M^{\I})}\mathbb{P}_{s_{\init},M^{\I}}^{\bm{\pi}^0}(\textup{reach}(C))\cdot\mathcal{P}^{\bm{\pi}^{C}}_{\I}(\mathcal{H}_{\D}\,|\,\textup{reach}(C))=1.
 \end{multline*}
By Theorem~\ref{thm:main}, APD is achieved.	
\end{proof}

\subsection{Proof of Lemma~\ref{lema:expconvergence}}\label{sec:proofexpconvergence}
Inspired by~\cite{DK:78}, we first present a compact formula to compute the BC for a binary MMDP under a stationary policy.
\begin{lema}[Computation of the BC via matrix multiplication]
Given a binary MMDP $\mathcal{M}=\{M_1,M_2\}$ and a stationary policy $\bm{\pi}$, the BC $B(t,\bm{\pi})$ can be computed by
\begin{equation}\label{eq:BCasmatrixproduct}
B(t,\bm{\pi})=\mathbbb{e}_{s_{\init}}^\top W^t\mathbbb{1}_n,
\end{equation}
where the $(i,j)$-th element $W_{ij}$ of the matrix  $W\in\mathbb{R}^{|\mathcal{S}|\times |\mathcal{S}|}$ is defined by
\begin{equation}\label{eq:matrixQ}
W_{ij}=\sum_{a\in\mathcal{A}_i}\bm{\pi}(a\,|\,i)\sqrt{\delta_1(j\,|\,i,a)\delta_2(j\,|\,i,a)}.
\end{equation}
\end{lema}
\begin{proof}
We first prove by induction that for any $s\in\mathcal{S}$ and $t\geq0$,
\begin{equation}\label{eq:product}
\mathbbb{e}_{s_{\init}}^\top W^t\mathbbb{e}_{s}=\sum_{h_t(0)=s_{\init},h_t(t)=s}\sqrt{\mathbb{P}_1^{\bm{\pi}}(h_t)\mathbb{P}_2^{\bm{\pi}}(h_t)}.
\end{equation}	
When $t=0$, \eqref{eq:product} holds, i.e., both sides of \eqref{eq:product} are one when $s=s_{\init}$ and zero otherwise. Suppose  \eqref{eq:product} holds for $t=\tau$. When $t=\tau +1$, we have
\begin{align*}
&\sum_{h_{\tau+1}(0)=s_0,h_{\tau+1}(\tau+1)=s}\sqrt{\mathbb{P}_1^{\bm{\pi}}(h_{\tau+1})\mathbb{P}_2^{\bm{\pi}}(h_{\tau+1})}\\
&=\sum_{s'\in\mathcal{S}}\sum_{h_{\tau}(0)=s_0,h_{\tau}(\tau)=s'}\sqrt{\mathbb{P}_1^{\bm{\pi}}(h_{\tau})\mathbb{P}_2^{\bm{\pi}}(h_{\tau})}\\
&\qquad\qquad\qquad\quad\cdot\sum_{a\in\mathcal{A}_{s'}}\bm{\pi}(a\,|\,s')\sqrt{\delta_{1}(s\,|\,s',a)\delta_{2}(s\,|\,s',a)}\\
&=\sum_{s'\in\mathcal{S}}\mathbbb{e}_{s_{\init}}^\top W^t\mathbbb{e}_{s'}\cdot W_{s's}\\
&=\mathbbb{e}_{s_{\init}}^\top W^{t+1}\mathbbb{e}_{s},
\end{align*}
where the first equality is due to the fact that the policy is stationary and Markovian, and the second equality follows from the induction hypothesis and the definition of $W$.

The conclusion~\eqref{eq:BCasmatrixproduct} then follows from the observation that
\begin{align*}
B(t,\bm{\pi})&=\sum_{h_t\in \mathcal{H}_t}\sqrt{\mathbb{P}_1^{\bm{\pi}}(h_t)\mathbb{P}_2^{\bm{\pi}}(h_t)}\\
&=\sum_{s\in\mathcal{S}}\sum_{h_t(0)=s_0,h_t(t)=s}\sqrt{\mathbb{P}_1^{\bm{\pi}}(h_t)\mathbb{P}_2^{\bm{\pi}}(h_t)}\\
&=\sum_{s\in\mathcal{S}}\mathbbb{e}_{s_{\init}}^\top W^t\mathbbb{e}_{s}=\mathbbb{e}_{s_{\init}}^\top W^t\mathbbb{1}_n.
\end{align*}
\end{proof}

Now we present the proof of Lemma~\ref{lema:expconvergence}.
\begin{proof}[Proof of Lemma~\ref{lema:expconvergence}]
We first note that the matrix $W$ defined by~\eqref{eq:matrixQ} is non-negative and has row sum less than or equal to one.  By~\cite[Theorem 4.11]{FB:20}, the spectral radius of $W$ is less than or equal to one.

Let the set of eigenvalues of $W$ be $\{\lambda_i\}_{i\in\mathbb{N}_{1}^{|\mathcal{S}|}}$ where $|\lambda_i|\leq1$ for all $i\in\mathbb{N}_{1}^{|\mathcal{S}|}$. Then, by the Jordan decomposition of $W$, we have
 \begin{equation}\label{eq:BCJordan}
\mathbbb{e}_{s_{\init}}^\top W^t\mathbbb{1}_{|\mathcal{S}|}=\sum_{i=1}^{|\mathcal{S}|}\sum_{k=0}^{t-1}c_{ik}t^k\lambda_i^{t-k},
\end{equation}
where $c_{ik}$ are constant coefficients. Since APD is achieved for $\mathcal{M}$ under the policy $\bm{\pi}$, by Lemma~\ref{lema:APDgeneral}, the BC must vanish, i.e., 
\begin{equation*}
	\lim_{t\rightarrow\infty}\mathbbb{e}_{s_{\init}}^\top W^t\mathbbb{1}_{|\mathcal{S}|}=\lim_{t\rightarrow\infty}\sum_{i=1}^{|\mathcal{S}|}\sum_{k=0}^{t-1}c_{ik}t^k\lambda_i^{t-k}=0.
\end{equation*}
By the form~\eqref{eq:BCJordan}, the BC must converge exponentially fast.
\end{proof}

\subsection{Proof of Lemma~\ref{lema:BCgeneral}}\label{appendix:BCgeneral}
Let $R_i=\{z\in\mathbb{R}^n\,|\,q_if_i(z)\geq q_jf_j(z) \textup{ for all }j\in\mathbb{N}_1^N\}$. Then the probability of error is given by
\begin{align*}
	P_{\textup{error}}&=\sum_{i=1}^N\theta_i	\sum_{j\neq i}\int_{R_j}f_i(z)dz\\
	&=\sum_{i=1}^N\frac{\theta_i}{q_i}	\sum_{j\neq i}\int_{R_j}q_if_i(z)dz\\
		&\leq\max_{i}\{\frac{\theta_i}{q_i}\}\cdot\sum_{i=1}^N	\sum_{j\neq i}\int_{R_j}q_if_i(z)dz.
\end{align*}
The upper bound on the probability of error then follows from~\cite[Theorem 6]{KM:71}.


Our proof of the lower bound is inspired by that for the binary case in~\cite[Appendix A]{TTK-LAS:67}. 
By the Cauchy-Schwarz inequality, for any measurable set $R$ in the Borel $\sigma$-algebra on $\mathbb{R}^n$ and $i,j\in\mathbb{N}_{1}^N$, we have
\begin{align}\label{eq:CS}
	\begin{split}
	\int_R \sqrt{f_i(z)f_j(z)}dz&\leq\sqrt{\int_R f_i(z)dz\int_R f_j(z)dz}\\
&\leq \sqrt{\int_R f_k(z)dz}~~~\textup{for any }k\in\{i,j\}.
	\end{split}
\end{align}
Therefore, we have
\begin{align}\label{eq:BijCS}
	\begin{split}
&\quad\min\{\theta_i,\theta_j\}	B_{ij}^2\\
&=\min\{\theta_i,\theta_j\}(\int_R \sqrt{f_i(z)f_j(z)}dz+\int_{\overline{R}} \sqrt{f_i(z)f_j(z)}dz)^2\\
&\leq\min\{\theta_i,\theta_j\} ( \sqrt{\int_R f_i(z)dz}+ \sqrt{\int_{\overline{R}} f_j(z)dz})^2\\
&\leq( \sqrt{\theta_i\int_R f_i(z)dz}+ \sqrt{\theta_j\int_{\overline{R}} f_j(z)dz})^2\\	
&=\theta_i\int_R f_i(z)dz+\theta_j\int_{\overline{R}} f_j(z)dz\\
&\quad+2\sqrt{\theta_i\theta_j\int_R f_i(z)dz\int_{\overline{R}} f_j(z)dz}\\
&\leq 2(\theta_i\int_R f_i(z)dz+\theta_j\int_{\overline{R}} f_j(z)dz),
	\end{split}
\end{align}
where the third line follows from~\eqref{eq:CS}. Fix any $k\in\mathbb{N}_{1}^N$, we then bound the probability of error from below as follows,
\begin{align*}
	P_{\textup{error}}
	&=\sum_{i=1}^N\theta_i	\sum_{j\neq i}\int_{R_j}f_i(z)dz\\
	&=\sum_{i\neq k}	\big(\theta_i\sum_{j\neq i}\int_{R_j}f_i(z)dz+\theta_k\int_{R_i}f_k(z)dz\big)\\
	&\geq \sum_{i\neq k}\frac{1}{2}\min\{\theta_i,\theta_k\}B_{ik}^2,
\end{align*}
where the last inequality follows from~\eqref{eq:BijCS}. Thus, we have the following lower bound for the probability of error,
\begin{equation*}
 P_{\textup{error}}\geq	\frac{1}{2}\max_{k\in\mathbb{N}_{1}^N}\{\sum_{i\neq k}\min\{\theta_i,\theta_k\}B_{ik}^2\}.
\end{equation*}
We note that the lower bound derived in~\cite[Section 2]{SK:76} also has the property that it is zero if and only if the pairwise BCs are zero and thus serves our purpose. 

\subsection{Proof of Theorem~\ref{thm:alggeneral}}\label{sec:proofalggeneral}

We first present a lemma that connects the MECs in the transition system $T$ with those in the informative MDPs.
\begin{lema}[MECs of the transition system and the informative MDPs]\label{lema:MECinTandMI}
Given an MMDP $\mathcal{M}=\{M_i\}_{i\in\mathcal{N}}$ and the transition system $T$ built by Algorithm~\ref{alg:APDgeneral}, if $C\in\mathcal{C}(T)$ is an MEC in $T$ and $C\notin\{(\{\bot_1^{\g}\},\{a^{\bot_1^{\g}}\}),(\{\bot_0^{\g}\},\{a^{\bot_0^{\g}}\})\}$, then $C$ is also an MEC in the informative MDP $M^{\I}_{ij}$ for any pair of MDPs $M_i,M_j\in\mathcal{M}$.
\end{lema}
\begin{proof}
Since $C=(\mathcal{X},\mathcal{U})\in\mathcal{C}(T)$ is an MEC in $T$ and $C\notin\{(\{\bot_1^{\g}\},\{a^{\bot_1^{\g}}\}),(\{\bot_0^{\g}\},\{a^{\bot_0^{\g}}\})\}$, for any $s,s'\in\mathcal{X}$ and $u\in\mathcal{U}_s$, we have that $\delta_{i'}(s'\,|\,s,u)>0$ for some $i'\in\mathcal
N$ implies that $\delta_{i}(s'\,|\,s,u)>0$ for all $i\in\mathcal{N}$. Moreover, we also have $\delta_{i}(\bot_j^{\g}\,|\,s,u)=0$ for $j\in\{0,1\}$ and $i\in\mathcal{N}$. Therefore, we conclude that $C$ is also an MEC in the informative MDP $M^{\I}_{ij}$ for any pair of MDPs $M_i$ and $M_j$.
\end{proof}
\begin{proof}[Proof of Theorem~\ref{thm:alggeneral}]

We prove the finite termination and correctness of Algorithm~\ref{alg:APDgeneral} by induction.

\subsubsection{Finite termination}
When the input MMDP $\{M_i\}_{i\in\mathcal{N}}$ is binary, i.e., $|\mathcal{N}|=2$, Algorithm~\ref{alg:APDgeneral} calls Algorithm~\ref{alg:APD}, and by Theorem~\ref{thm:alg}, we conclude that Algorithm~\ref{alg:APDgeneral} terminates in finite time.  Suppose Algorithm~\ref{alg:APDgeneral} terminates in finite time when the input MMDP $\{M_i\}_{i\in\mathcal{N}}$ consists of $|\mathcal{N}|\leq K$ MDPs. For an input MMDP $\{M_i\}_{i\in\mathcal{N}}$ with $|\mathcal{N}|=K+1$ MDPs, if Algorithm~\ref{alg:APDgeneral} does not call itself, then we are in the situation discussed in Section~\ref{sec:basecase} and the modified Algorithm~\ref{alg:APD} terminates in finite time. Otherwise, by the induction hypothesis, all the recursive calls terminate in finite time and there are finitely many of them. Thus, we conclude that 
Algorithm~\ref{alg:APDgeneral} terminates in finite time when the input MMDP consists of $|\mathcal{N}|=K+1$ MDPs.

\subsubsection{Correctness} When the input MMDP $\{M_i\}_{i\in\mathcal{N}}$ is binary, i.e., $|\mathcal{N}|=2$, Algorithm~\ref{alg:APDgeneral} calls Algorithm~\ref{alg:APD}, and by Theorem~\ref{thm:alg}, we conclude that Algorithm~\ref{alg:APDgeneral} is correct.
	
Suppose Algorithm~\ref{alg:APDgeneral} is correct when the input MMDP $\{M_i\}_{i\in\mathcal{N}}$ consists of $|\mathcal{N}|\leq K$ MDPs. 
			
Let $\{M_i\}_{i\in\mathcal{N}}$ be an input MMDP with $|\mathcal{N}|=K+1$ MDPs. If Algorithm~\ref{alg:APDgeneral} dose not call itself, then by Lemma~\ref{lema:basecase}, Algorithm~\ref{alg:APDgeneral} is correct. Otherwise, we show that a policy that achieves APD for $\{M_i\}_{i\in\mathcal{N}}$ exists if and only if $s_{\init}\in\mathcal{R}^{\max}$, where $\mathcal{R}^{\max}$ is defined in line~\ref{alg2:Rmax} of Algorithm~\ref{alg:APDgeneral}.

Suppose $s_{\init}\in\mathcal{R}^{\max}$. Note that the set of informative MECs $\mathcal{C}^{\I}(T)$ of the transition system $T$ consists of two types of states: i) $A_1=\{\bot^{\g}_1\}$; ii) $A_2=\{s\in\mathcal{S}\,|\,s\in\mathcal{C}^{\I}(T)\}$. Since $s_{\init}\in\mathcal{R}^{\max}$, following the policy $\bm{\pi}_{\mathcal{N}}^0$ guarantees that the union of $A_1$ and $A_2$ is reached with probability $1$ in $T$. Let $M_i, M_j\in\{M_i\}_{i\in\mathcal{N}}$ be an arbitrary pair of MDPs. Then, By Lemma~\ref{lema:MECinTandMI} and the definition of $\mathcal{C}^{\I}(T)$ in~\eqref{eq:informativeMECbase}, we know that the set of states $A_2$ also constitute informative MECs in $M^{\I}_{ij}$. On the other hand, reaching $A_1$ in $T$ could correspond to the following situations in $M^{\I}_{ij}$:
\begin{enumerate}[label=(\roman*)]
	\item if $(s,a,\bot^{\g}_1)$ comes from the transition $(s,a,s')$ during the BFS where  $\delta_{i}(s'\,|\,s,a)\delta_{j}(s'\,|\,s,a)=0$, then $s'$ is reachable from $(s,a)$ in at most one of $M_i$ and $M_j$. In this case, the transition $(s,a,s')$ either does not exist in $M^{\I}_{ij}$ or is replaced by $(s,a,\bot_k)$ for $k\in\{1,2\}$ in $M^{\I}_{ij}$. 
	\item if $(s,a,\bot^{\g}_1)$ comes from the transition $(s,a,s')$ during the BFS where  $\delta_{i}(s'\,|\,s,a)\delta_{j}(s'\,|\,s,a)>0$, then, by the induction hypothesis, $\bot^{\g}_1$ indicates that there exists a policy such that APD is achieved for the set of MDPs consisting of $M_i$ and $M_j$.
\end{enumerate}
In all of the above cases, we have that APD is achieved for the pair $M_i$ and $M_j$ in $\{M_i\}_{i\in\mathcal{N}}$. Then, by Lemma~\ref{lema:BCgeneral}, APD is achieved for $\{M_i\}_{i\in\mathcal{N}}$.

Suppose $s_{\init}\notin\mathcal{R}^{\max}$. Then, following any policy in $T$ from $s_{\init}$ results in a strictly positive probability of reaching the union of the states in $A_1'$ and $A_2'$ where $A_1'=\{\bot_0^{\g}\}$ and $A_2'=\{s\in\mathcal{S}\,|\,s\in\mathcal{C}(T)\setminus\mathcal{C}^{\I}(T)\}$. We consider two scenarios:
\begin{enumerate}[label=(\roman*)]
\item suppose following any policy in $T$, the probability of reaching an MEC $C$ composed of states in $A_2'$ from $s_{\init}$  is strictly positive. Consider the pair of MDPs $M_i$ and $M_j$ such that $C$ is not an informative MDP in $M^{\I}_{ij}$ (such a pair must exist, otherwise $C\in\mathcal{C}^{\I}(T)$). Then, in the informative MDP $M_{ij}^{\I}$, the probability of reaching the set of informative MDPs must be strictly less than $1$, which, by Theorem~\ref{thm:alg}, implies that there does not exist a policy that achieves APD for $M_i$ and $M_j$; 
\item suppose following any policy in $T$, the probability of reaching $A_1'$ from $s_{\init}$ is strictly positive. If $(s,a,\bot_{0}^{\g})$ replaces the transition $(s,a,s')$ as the result of the recursive call $\texttt{APD}(\{M_i\}_{i\in\mathcal{N}'},s',\bm{\Pi})$, then by the induction hypothesis, there does not exist policies under which APD is achieved for $\{M_i\}_{i\in\mathcal{N}'}$. In other words, there exists a pair of $M_i$ and $M_j$ for $i,j\in\mathcal{N}'$ such that $B_{ij}(\bm{\pi})>0$ for any $\bm{\pi}$.
\end{enumerate}
In all above cases, there does not exist a policy that achieves APD for $\{M_i\}_{i\in\mathcal{N}}$.

\end{proof}

\bibliographystyle{IEEEtran}
\bibliography{alias,mybib}

\begin{thebibliography}{10}
\providecommand{\url}[1]{#1}
\csname url@samestyle\endcsname
\providecommand{\newblock}{\relax}
\providecommand{\bibinfo}[2]{#2}
\providecommand{\BIBentrySTDinterwordspacing}{\spaceskip=0pt\relax}
\providecommand{\BIBentryALTinterwordstretchfactor}{4}
\providecommand{\BIBentryALTinterwordspacing}{\spaceskip=\fontdimen2\font plus
\BIBentryALTinterwordstretchfactor\fontdimen3\font minus
  \fontdimen4\font\relax}
\providecommand{\BIBforeignlanguage}[2]{{%
\expandafter\ifx\csname l@#1\endcsname\relax
\typeout{** WARNING: IEEEtran.bst: No hyphenation pattern has been}%
\typeout{** loaded for the language `#1'. Using the pattern for}%
\typeout{** the default language instead.}%
\else
\language=\csname l@#1\endcsname
\fi
#2}}
\providecommand{\BIBdecl}{\relax}
\BIBdecl

\bibitem{MLP:14}
M.~L. Puterman, \emph{Markov Decision Processes: Discrete Stochastic Dynamic
  Programming}.\hskip 1em plus 0.5em minus 0.4em\relax John Wiley \& Sons,
  2014.

\bibitem{LNS-DLK-BTD:21}
L.~N. Steimle, D.~L. Kaufman, and B.~T. Denton, ``Multi-model {Markov} decision
  processes,'' \emph{IISE Transactions}, vol.~53, no.~10, pp. 1124--1139, 2021.

\bibitem{BW-MA-SB-UT:19}
B.~Wu, M.~Ahmadi, S.~Bharadwaj, and U.~Topcu, ``Cost-{Bounded} {Active}
  {Classification} {Using} {Partially} {Observable} {Markov} {Decision}
  {Processes},'' in \emph{{A}merican {C}ontrol {C}onference}, Philadelphia, PA,
  USA, Jul. 2019, pp. 1216--1223.

\bibitem{KC-MC-DK-PN-AR:20}
K.~Chatterjee, M.~Chmelík, D.~Karkhanis, P.~Novotn\'{y}, and A.~Royer,
  ``Multiple-{Environment} {Markov} {Decision} {Processes}: {Efficient}
  {Analysis} and {Applications},'' in \emph{{International} {Conference} on
  {Automated} {Planning} and {Scheduling}}, Nancy, France, Oct. 2020, pp.
  48--56.

\bibitem{GS-DH-RIB:05}
G.~Shani, D.~Heckerman, and R.~I. Brafman, ``An {MDP}-{Based} {Recommender}
  {System},'' \emph{Journal of Machine Learning Research}, vol.~6, no.~43, pp.
  1265--1295, 2005.

\bibitem{IC-JC-TGM-SN-RS-OB:12}
I.~Chad\`{e}s, J.~Carwardine, T.~G. Martin, S.~Nicol, R.~Sabbadin, and
  O.~Buffet, ``{MOMDPs}: {A} {Solution} for {Modelling} {Adaptive} {Management}
  {Problems},'' in \emph{{AAAI} {Conference} on {Artificial} {Intelligence}},
  Toronto, Ontario, Canada, Jul. 2012, pp. 267--273.

\bibitem{EB-LL:13}
E.~Brunskill and L.~Li, ``Sample complexity of multi-task reinforcement
  learning,'' in \emph{{The} {Conference} on {Uncertainty} in {Artificial}
  {Intelligence}}, Arlington, Virginia, USA, Aug. 2013, pp. 122--131.

\bibitem{JR-OS:14}
\BIBentryALTinterwordspacing
J.~Raskin and O.~Sankur, ``Multiple-{Environment} {Markov} {Decision}
  {Processes},'' \emph{arXiv}, 2014. [Online]. Available:
  \url{http://arxiv.org/abs/1405.4733}
\BIBentrySTDinterwordspacing

\bibitem{AH-DDC-SM:15}
\BIBentryALTinterwordspacing
A.~Hallak, D.~D. Castro, and S.~Mannor, ``Contextual {Markov} {Decision}
  {Processes},'' \emph{arXiv}, 2015. [Online]. Available:
  \url{https://arxiv.org/abs/1502.02259v1}
\BIBentrySTDinterwordspacing

\bibitem{PB-DS:19}
P.~Buchholz and D.~Scheftelowitsch, ``Computation of weighted sums of rewards
  for concurrent {MDPs},'' \emph{Mathematical Methods of Operations Research},
  vol.~89, no.~1, pp. 1--42, 2019.

\bibitem{JK-YE-CC-SM:21}
\BIBentryALTinterwordspacing
J.~Kwon, Y.~Efroni, C.~Caramanis, and S.~Mannor, ``{RL} for {Latent} {MDPs}:
  {Regret} {Guarantees} and a {Lower} {Bound},'' \emph{arXiv}, 2021. [Online].
  Available: \url{https://arxiv.org/abs/2102.04939}
\BIBentrySTDinterwordspacing

\bibitem{KC-MC-MT:16}
K.~Chatterjee, M.~Chmel\'{i}k, and M.~Tracol, ``What is decidable about
  partially observable {Markov} decision processes with $\omega$-regular
  objectives,'' \emph{Journal of Computer and System Sciences}, vol.~82, no.~5,
  pp. 878--911, 2016.

\bibitem{BCL:08}
B.~C. Levy, \emph{Principles of Signal Detection and Parameter
  Estimation}.\hskip 1em plus 0.5em minus 0.4em\relax Springer, 2008.

\bibitem{MSB:51}
M.~S. Bartlett, ``The frequency goodness of fit test for probability chains,''
  \emph{Mathematical Proceedings of the Cambridge Philosophical Society},
  vol.~47, no.~1, pp. 86--95, 1951.

\bibitem{TWA-LAG:57}
T.~W. Anderson and L.~A. Goodman, ``Statistical {Inference} about {Markov}
  {Chains},'' \emph{The Annals of Mathematical Statistics}, vol.~28, no.~1, pp.
  89--110, 1957.

\bibitem{DK:78}
D.~Kazakos, ``The {Bhattacharyya} distance and detection between {Markov}
  chains,'' \emph{IEEE Transactions on Information Theory}, vol.~24, no.~6, pp.
  747--754, 1978.

\bibitem{CD-ND-NG:18}
C.~Daskalakis, N.~Dikkala, and N.~Gravin, ``Testing {Symmetric} {Markov}
  {Chains} {From} a {Single} {Trajectory},'' in \emph{Conference {On}
  {Learning} {Theory}}, Stockholm, Sweden, Jul. 2018.

\bibitem{YC-PLB:19}
Y.~Cherapanamjeri and P.~L. Bartlett, ``Testing {Symmetric} {Markov} {Chains}
  {Without} {Hitting},'' in \emph{Conference on {Learning} {Theory}}, Phoenix,
  AZ, USA, Jun. 2019.

\bibitem{GW-AK:20}
G.~Wolfer and A.~Kontorovich, ``Minimax {Testing} of {Identity} to a
  {Reference} {Ergodic} {Markov} {Chain},'' in \emph{International {Conference}
  on {Artificial} {Intelligence} and {Statistics}}, Palermo, Sicily, Italy,
  Jun. 2020, pp. 191--201.

\bibitem{JKS-REL:73}
J.~K. Satia and R.~E.~L. Jr., ``Markovian decision processes with uncertain
  transition probabilities,'' \emph{Operations Research}, vol.~21, no.~3, pp.
  661--865, 1973.

\bibitem{CCW-HKE_94}
C.~C.~W. III and H.~K. Eldeib, ``Markov decision processes with imprecise
  transition probabilities,'' \emph{Operations Research}, vol.~42, no.~4, pp.
  574--788, 1994.

\bibitem{GNI:05}
G.~N. Iyengar, ``Robust {Dynamic} {Programming},'' \emph{Mathematics of
  Operations Research}, vol.~30, no.~2, pp. 257--280, 2005.

\bibitem{AN-LEIG:05}
A.~Nilim and L.~E. Ghaoui, ``Robust {Control} of {Markov} {Decision}
  {Processes} with {Uncertain} {Transition} {Matrices},'' \emph{Operations
  Research}, vol.~53, no.~5, pp. 780--798, 2005.

\bibitem{WW-DK-BR:13}
W.~Wiesemann, D.~Kuhn, and B.~Rustem, ``Robust {Markov} {Decision}
  {Processes},'' \emph{Mathematics of Operations Research}, vol.~38, no.~1, pp.
  153--183, 2013.

\bibitem{AB:46}
A.~Bhattacharyya, ``On a {Measure} of {Divergence} between {Two} {Multinomial}
  {Populations},'' \emph{Sankhy\={a}: The Indian Journal of Statistics},
  vol.~7, no.~4, pp. 401--406, 1946.

\bibitem{CB-JK:08}
C.~Baier and J.~Katoen, \emph{Principles of Model Checking}.\hskip 1em plus
  0.5em minus 0.4em\relax MIT Press, 2008.

\bibitem{KC-MH:11}
K.~Chatterjee and M.~Henzinger, ``Faster and dynamic algorithms for maximal
  end-component decomposition and related graph problems in probabilistic
  verification,'' in \emph{{ACM}-{SIAM} symposium on {Discrete} algorithms},
  San Francisco, California, USA, Jan. 2011, pp. 1318--1336.

\bibitem{TTK-LAS:67}
T.~T. Kadota and L.~A. Shepp, ``On the best finite set of linear observables
  for discriminating two {Gaussian} signals,'' \emph{IEEE Transactions on
  Information Theory}, vol.~13, no.~2, pp. 278--284, 1967.

\bibitem{SA:15}
S.~Abbott, \emph{Understanding Analysiss}, 2nd~ed., ser. Undergraduate Texts in
  Mathematics.\hskip 1em plus 0.5em minus 0.4em\relax Springer, 2015.

\bibitem{RBA-CAD:99}
R.~B. Ash and C.~A. Dol\'{e}ans-Dade, \emph{Probability \& measure theory},
  2nd~ed.\hskip 1em plus 0.5em minus 0.4em\relax Cambridge, MA: Academic Press,
  1999.

\bibitem{SK:48}
S.~Kakutani, ``On {Equivalence} of {Infinite} {Product} {Measures},''
  \emph{Annals of Mathematics}, vol.~49, no.~1, pp. 214--224, 1948.

\bibitem{LdeA:97}
L.~de~Alfaro, ``Formal verification of probabilistic systems,'' Ph.D.
  dissertation, Stanford University, 1997.

\bibitem{FB:20}
\BIBentryALTinterwordspacing
F.~Bullo, \emph{Lectures on Network Systems}, {1.4}~ed.\hskip 1em plus 0.5em
  minus 0.4em\relax Kindle Direct Publishing, Jul. 2020, with contributions by
  J. Cort{\'e}s, F. D\"orfler, and S. Mart{\'\i}nez. [Online]. Available:
  \url{http://motion.me.ucsb.edu/book-lns}
\BIBentrySTDinterwordspacing

\bibitem{KM:71}
K.~Matusita, ``Some properties of affinity and applications,'' \emph{Annals of
  the Institute of Statistical Mathematics}, vol.~23, no.~1, pp. 137--155,
  1971.

\bibitem{SK:76}
S.~Kirmani, ``A lower bound on bayes risk in classification problems,''
  \emph{Annals of the Institute of Statistical Mathematics}, vol.~28, no.~1,
  pp. 385--387, 1976.

\end{thebibliography}

\end{document}